\documentclass[11pt,reqno]{amsart}

\usepackage{hyperref}
\usepackage{morefloats}

\usepackage{assorted}
\def\BL{\mathbb{L}}

\newcommand{\bluefootnote}[1]{
\relax
}
\DeclareRobustCommand{\note}[1]{%
\relax
}

\newcommand{\defnfont}[1]{\textbf{#1}}

\usepackage[backend=bibtex,hyperref=true,maxbibnames=5,style=numeric,firstinits=true,sorting=nty]{biblatex}
\makeatletter
\renewcommand{\theenumi}{\roman{enumi}}

\renewcommand{\p@enumii}{\theenumi.}
\makeatother

\title[On classification of modular categories]{On classification of modular categories by rank}
\author{Paul Bruillard}
\email{pjb2357@gmail.com}
\address{Department of Mathematics,
    Texas A\&M University,
    College Station, TX
    U.S.A.}
\curraddr{Pacific Northwest National Laboratory, 902 Battelle Boulevard, Richland, WA U.S.A}
\thanks{\textit{PNNL Information Release:} PNNL-SA-111550.}

\author{Siu-Hung Ng}
\email{rng@math.lsu.edu}

\address{Department of Mathematics, Louisiana State University, Baton Rouge, LA
    U.S.A.}
\thanks{The first, third and fourth authors were partially supported by
NSF grant DMS1108725, and the second author was partially supported by NSF grants
DMS1001566, DMS1303253, and DMS1501179. This project began at a workshop at the American Institute of Mathematics, whose support and hospitality are gratefully acknowledged. It is a pleasure to thank Vladimir Turaev and Pavel Etingof for valuable comments.}

\author{Eric C. Rowell}
\email{rowell@math.tamu.edu}
\address{Department of Mathematics,
    Texas A\&M University,
    College Station, TX
    U.S.A.}

\author{Zhenghan Wang}
\email{zhenghwa@microsoft.com}
\address{Microsoft Research Station Q and Department of Mathematics,
    University of California,
    Santa Barbara, CA
    U.S.A.}


\date{\today}
\dedicatory{}

\begin{document}
\raggedbottom
\begin{abstract}
The feasibility of a classification-by-rank program for modular categories follows from the Rank-Finiteness Theorem.
We develop arithmetic, representation theoretic and algebraic methods for classifying modular
categories by rank. As an application, we determine all
possible fusion rules for all rank=$5$ modular categories and describe the
corresponding monoidal equivalence classes.
\end{abstract}

\maketitle

\section{Introduction}\label{Introduction}
Modular categories arise in a variety of mathematical subjects including topological quantum field theory
\cite{Tu1}, conformal field theory \cite{MS1}, representation theory of quantum groups \cite{BKi}, von Neumann algebras \cite{EK1}, and vertex operator algebras \cite{Hu1}.  They are quantum analogues of finite groups as illustrated by the Cauchy and Rank-Finiteness theorems \cite{BNRW1}.
Classification of low-rank modular categories is a first step in a structure theory for modular categories parallel to group theory.
Besides the intrinsic mathematical aesthetics, another motivation for pursuing a classification of modular categories comes from their application in topological phases of matter and topological quantum computation \cite{W1, W2}.  A classification of modular categories is literally a classification of certain topological phases of matter \cite{WW1, BCFV1}.

The first success of the classification program was the complete classification of unitary modular categories up to rank=$4$ in \cite{RSW}.  That such a program is theoretically feasible follows from the Rank-Finiteness Theorem (which we proved for modular categories in \cite{BNRW1}, and extended to pre-modular categories in \cite[Appendix]{B2}): there are only finitely many inequivalent modular categories of a given rank $r$.
In this paper, we develop arithmetic, representation theoretic and algebraic tools for a classification-by-rank program.  As an application we complete a classification of all modular categories of rank=$5$ (up to monoidal equivalence) in Section \ref{Applications: Rank 5 Classification}.

A modular category $\mathcal{C}$ is a non-degenerate ribbon fusion category over $\mathbb{C}$
  \cite{Tu1, BKi}.  Let $\Pi_\CC$ be the set of isomorphism classes of simple objects of the modular category $\CC$. The \defnfont{rank} of $\CC$ is the finite number $r=|\Pi_\CC|$.  Each modular category $\mathcal{C}$ leads to a $(2+1)$-dimensional
  topological  quantum field theory $(V_\mathcal{C}, Z_\mathcal{C})$, in particular colored
  framed link invariants \cite{Tu1}.  The invariant $\{d_a\}$ for the unknot
  colored by the label $a\in \Pi_\CC$ is called the {\it quantum
  dimension} of the label.
  The invariant of the Hopf
  link colored by $a,b$ will be denoted as $S_{ab}$. The link invariant of the unknot with a right-handed
  kink colored by $a$ is $\theta_a\cdot d_a$ for some root of unity $\theta_a$,
  which is called the {\it topological twist} of the label $a$.  The topological
  twists form a diagonal matrix $T=(\delta_{ab}\theta_a), a,b\in \Pi_\CC$.   The
  $S$-matrix and $T$-matrix together lead to a projective representation of the
  modular group ${\rm SL}(2,\Z)$ by sending the generating matrices

  \begin{equation*}
     \fs=\mtx{0 & -1\\ 1 & 0},\qquad \ft=\mtx{1 & 1\\ 0 & 1}
  \end{equation*}
  to $S,T$, respectively \cite{Tu1, BKi}.  Amazingly, the kernel of this projective representation of $\CC$ is always a congruence
  subgroup of ${\rm SL}(2,\Z)$ \cite{NS3}.  The $S$-matrix determines the fusion rules through the Verlinde formula, and the $T$-matrix is of finite order $\ord(T)$ by
  Vafa's theorem \cite{BKi}.  Together, the pair $S,T$ are called the \emph{modular data} of the category $\CC$.

Modular categories are fusion categories with additional braiding and
  pivotal structures \cite{ENO1,Tu1,BKi}.  These extra structures endow them with some
  \lq\lq abelian-ness" which makes the classification of modular categories easier.
  The abelian-ness of modular categories first
  manifests itself in the braiding: the tensor product is commutative up to
  functorial isomorphism.  But a deeper sense of
  abelian-ness is revealed in the
  Galois group of  the number field $\mbbK_{\mathcal{C}}$ obtained by adjoining all
  matrix entries of $S$ to $\mathbb{Q}$: $\mbbK_{\mathcal{C}}$ is an abelian
  extension of $\mathbb{Q}$ \cite{BG1, RSW}.  Moreover, its Galois group is
  isomorphic to an abelian subgroup of the symmetric group $\mathfrak{S}_r$, where $r$ is the rank of
  $\mathcal{C}$.  This profound observation permits the application of deep number theory to
  the classification of modular categories.

The content of the paper is as follows.  Section \ref{Modular Categories} is a collection of necessary results on fusion and modular categories.  We define admissible modular data as a pair  of matrices $S,T$ satisfying algebraic constraints with an eye towards the characterization of realizable modular data.

In Section \ref{Arithmetic Properties of Modular Categories} we develop general arithmetic constraints on admissible modular data.  One improvement to the approach in \cite{RSW} is the combining of Galois symmetry of $S,T$ matrices with the knowledge of the representation theory of ${\rm SL}(2,\mathbb{Z})$.  An important observation is:

\begin{replemma}{l:1}
    Let $\CC$ be a modular category of rank $r$ and $\rho: {\rm SL}(2,\Z) \to \GLC{r}$ a
    modular representation of $\CC$, i.e. a lifting of projective representation of
    $\CC$. Then $\rho$ cannot be isomorphic to a direct sum of two representations
    with disjoint $\ft$-spectra.
\end{replemma}

Finally in Section \ref{Applications: Rank 5 Classification}, we combine the analysis of Galois action on the $S$-matrix and
  ${\rm SL}(2,\Z)$ representation to determine all possible fusion rules for all rank=$5$
  modular categories and describe their classification up to monoidal equivalence.

Our main result is:

   \begin{reptheorem}{thm:fusionrules}
      Suppose $\mcC$ is a modular category of rank $5$.  Then $\mcC$ is Grothendieck
      equivalent to  one of the following:
      \begin{enumerate}
        \item $SU(2)_4$
        \item $SU(2)_9/\Z_2$
        \item $SU(5)_1$
        \item $SU(3)_4/\Z_3$
      \end{enumerate}
    \end{reptheorem}

In this paper we only classify these modular categories up to monoidal equivalence, but a complete list of all modular categories with the above fusion rules as done in \cite{RSW} is possible.  However, the details are not straightforward, so we will leave it to a future publication.

A complete classification of the low-rank cases provides general insight for structure theory of modular categories. Such a classification is also very useful for the theory of topological phases of matter, and could shed light on the open problem that whether or not there are exotic modular categories, i.e., modular categories that are not closely related to the well-known quantum group construction \cite{HRW1}.  However, complete classification beyond rank=$5$ seems to be very difficult.

Topological phases of matter are phases of matter that lie beyond Landau's
symmetry breaking and local order parameter paradigm for the classification of states
  of matter.  Physicists propose to use the $S,T$ matrices as order parameters
  for the classification of topological phases of matter \cite{LWWY1}. Therefore, a natural
  question is if the $S,T$ matrices determine the modular category.  We believe they do.  The $S,T$ matrices satisfy many constraints, and a
  pair of matrices $S,T$ with those constraints are called  {\it admissible modular
  data}. It is interesting to characterize admissible modular data that can be realized by modular
  categories.

Modular categories form part of the mathematical foundations of topological quantum computation.
The classification program of modular categories initiated in this paper will lead to a deeper understanding of their structure and their enchanting relations to other fields, thus pave the way for applications to a futuristic field \emph{anyonics} broadly defined as the science and technology that cover the development, behavior, and application of anyonic devices.

\section{Modular Categories}
  \label{Modular Categories}
We follow the same conventions for modular categories as in \cite{BNRW1}. Most of the results below can be
found in  \cite{Tu1, BKi, ENO1, NS1, NS2, NS3, BNRW1} and the references
therein.  All fusion and modular categories are over the complex numbers $\mathbb{C}$ in this paper unless stated otherwise.

\subsection{Basic Invariants}
\label{subsection: Further Consequences and Properties}
  \subsubsection{Grothendieck Ring and Dimensions}
    The \defnfont{Grothendieck ring} $K_0(\CC)$ of a fusion category $\CC$ is the
    $\BZ$-ring generated by $\Pi_\CC$ with multiplication induced from
    $\o$.  The structure coefficients of $K_0(\CC)$ are obtained from:
    \begin{equation*}
      V_i \o V_j \cong \bigoplus_{k \in \Pi_\CC} N_{i,j}^k \,V_k
    \end{equation*}

    where $N_{i,j}^k = \dim(\Hom_{\mcC}\(V_k, V_i\o V_j\))$. This family of non-negative
    integers $\{N_{i,j}^k\}_{i,j,k \in \Pi_\CC}$ is called the \textit{fusion
    rules} of $\CC$.

    In a braided fusion category, $K_0(\CC)$ is a commutative ring and the fusion
    rules satisfy the symmetries:
    \begin{equation}
      \label{fusion symmetries}
      N_{i,j}^k=N_{j,i}^k=N_{i,k^*}^{j^*}=N_{i^*,j^*}^{k^*},\quad
      N_{i,j}^0=\delta_{i,j^*}.
    \end{equation}

    The \defnfont{fusion matrix} $N_i$ associated to $V_i$, defined by
    $(N_i)_{k,j}=N_{i,j}^k$, is
    an integral matrix with non-negative entries. In the braided fusion setting,
    these matrices are normal and mutually commuting. The largest real eigenvalue of
    $N_i$ is called the \defnfont{Frobenius-Perron
    dimension} of $V_i$ and is denoted by $\FPdim(V_i)$.
    Moreover, $\FPdim$ can be extended to a $\BZ$-ring homomorphism
    from $K_0(\CC)$ to $\BR$ and is the unique such homomorphism that is positive
    (real-valued) on $\Pi_\CC$ (see \cite{ENO1}). The \defnfont{Frobenius-Perron dimension} of
    $\CC$ is defined as
    \begin{equation*}
      \FPdim(\CC) = \sum_{i \in \Pi_\CC} \FPdim(V_i)^2\,.
    \end{equation*}

    \begin{defn}
      A fusion category $\mcC$ is said to be
      \begin{enumerate}
      \item \defnfont{weakly integral} if $\FPdim\(\mcC\)\in\mbbZ$.
        \item \defnfont{integral} if $\FPdim\(V_{j}\)\in\mbbZ$ for all
        $j\in\P_{\mcC}$.
        \item \defnfont{pointed} if $\FPdim\(V_{j}\)=1$ for all
        $j\in\P_{\mcC}$.
      \end{enumerate}

      Furthermore, if $\FPdim\(V\)=1$, then $V$ is \defnfont{invertible}.
    \end{defn}

    Let $\CC$ be a pivotal category. It follows from \cite[Prop. 2.9]{ENO1} that
    $d^{r}(V\du)=\ol{d^{r}(V)}$ is an algebraic integer for any $V \in \CC$. The
    \defnfont{global dimension} of $\CC$ is defined by
    \begin{equation*}
      D^2 = \sum_{i \in \Pi_\CC} |d^{r}(V_i)|^2.
    \end{equation*}

    By \cite{M2, ENO1}, a pivotal structure of a fusion category $\CC$ is spherical
    if, and only if, $d^{r}(V)$ is real for all $V \in \CC$. In this case,
    $d^{r}\(V\)=d^{\ell}\(V\)$ and we simply write $d\(V\)$ to refer to the
    dimension of $V$. Furthermore for $i\in\P_{\mcC}$, we adopt the shorthand
    $d_{i}=d\(V_{i}\)$.

    A fusion category $\CC$ is called \defnfont{pseudo-unitary} if $D^{2}=\FPdim(\CC)$.
    For a pseudo-unitary fusion category $\CC$, it has been
    shown in \cite{ENO1} that there exists a unique spherical structure of $\CC$
    such that $d\(V\) = \FPdim(V)$ for all objects $V \in \CC$.

  \subsubsection{Spherical and Ribbon Structures}

The set of isomorphism classes of invertible objects $G(\CC)$ in a fusion category $\CC$ forms a group in
    $K_0(\CC)$ where $i\inv = i^*$ for $i \in G(\CC)$. For modular categories $\CC$, the group
    $G(\CC)$ parameterizes pivotal structures on the underlying braided
    fusion category. Under such a correspondence, the
      inequivalent spherical structures of $\CC$ map onto the maximal elementary
      abelian 2-subgroup, $\Omega_2 G(\CC)$, of $G(\CC)$ \cite{BNRW1}.

In any ribbon fusion category $\CC$ the associated ribbon structure, $\theta$, has finite order.
    This celebrated fact is part of Vafa's Theorem
    (see \cite{Va1, BKi}) in the case of modular categories.
    However, any ribbon category embeds in a modular category
    (via Drinfeld centers, see \cite{M2}) so the result holds generally. Observe
    that, $\theta_{V_i} = \th_i \id_{V_i}$ for some root of unity $\th_i \in \BC$.
    Since $\theta_\1 = \id_\1$, $\th_0 = 1$. The $T$-\defnfont{matrix} of
    $\CC$ is defined by $T_{ij} = \delta_{ij} \th_j$ for $i,j \in \Pi_\CC$. The
    \textbf{balancing equation}:
    \begin{equation}
      \label{Balancing}
      \theta_i\theta_j S_{ij}=\sum_{k\in\Pi_\CC} N_{i^*j}^kd_k\theta_k
    \end{equation}

    is a useful algebraic consequence, holding in any premodular category.
    The pair $(S,T)$ of $S$ and $T$-matrices will be called the \defnfont{modular
    data} of a given modular category $\CC$.

  \subsubsection{Modular Data and ${\rm SL}(2,\Z)$ Representations}
   \begin{definition}
      For a pair of matrices $(S,T)$ for
      which there exists a modular category with modular data $(S,T)$, we will say
      $(S,T)$ is \defnfont{realizable modular data}.
    \end{definition}

    The fusion rules $\{N_{i,j}^k\}_{i,j,k \in \Pi_\CC}$ of
    $\CC$ can be written in terms of the $S$-matrix, via the
    \defnfont{Verlinde formula} \cite{BKi}:
    \begin{equation}
      \label{Verlinde Formula}
      N_{i,j}^k = \frac{1}{D^{2}} \sum_{a \in \Pi_\CC} \frac{S_{ia} S_{ja}
      S_{k^*a}}{S_{0a}} \quad\text{for all } i,j,k \in \Pi_\CC\,.
    \end{equation}

    The  modular data $(S, T)$ of a modular category $\CC$ satisfy the conditions:
    \begin{equation}
      \label{eq:STrelations}
      (ST)^3 = p^{+} S^2, \quad S^2=p^{+}p^{-} C, \quad CT=TC,\quad C^2=\id,
    \end{equation}

    where $p^{\pm} = \sum_{i\in \Pi_\CC} d_i^2 \th_i^{\pm 1} $ are called the
    \defnfont{Gauss sums}, and
    $C=\(\delta_{ij^*}\)_{i,j \in \Pi_\CC}$ is called the \defnfont{charge conjugation matrix} of $\CC$.
    In terms of matrix entries, the first equation in (\ref{eq:STrelations})
   gives the \textbf{twist equation}:
    \begin{equation}\label{twisteq}
      p^{+} S_{jk} = \theta_j \theta_k \sum_{i} \theta_i S_{ij} S_{ik}\,.
    \end{equation}
    The quotient $\frac{p^{+}}{p^{-}}$, called the \defnfont{anomaly} of $\CC$,
    is a  root of unity, and
    \begin{equation}
      \label{eq:gausssum}
      p^{+} p^{-}  =D^{2}.
    \end{equation}

    Moreover, $S$ satisfies
    \begin{equation}
      \label{eq:S}
      S_{ij} = S_{ji} \quad \text{and}\quad S_{ij^*} = S_{i^*j}
    \end{equation}

    for all $i, j \in \Pi_\CC$. These equations and the Verlinde formula imply that
    \begin{equation}
      \label{eq:orthogonality}
      S_{ij\du} = \ol{S_{ij}} \quad\text{and}\quad\frac{1}{D^{2}}\sum_{j \in \Pi_\CC}
      S_{ij}\ol{S}_{jk} = \delta_{ik}.
    \end{equation}

    In particular, $S$ is projectively unitary.

    A modular category $\CC$ is called \defnfont{self-dual} if $i=i\du$ for all $i \in
    \Pi_\CC$. In fact, $\mcC$ is self-dual if and only if $S$ is a real matrix.

    Let $D$ be the positive square
    root of $D^{2}$. The Verlinde formula can be rewritten as
    \begin{equation*}
      S N_i S\inv = D_i \quad \text{for }i \in \Pi_\CC
    \end{equation*}

    where $\(D_i\)_{ab} = \delta_{ab} \frac{S_{ia}}{S_{0a}}$. In particular, the
    assignments $\phi_a: i \mapsto \frac{S_{ia}}{S_{0a}}$ for $i \in \Pi_\CC$
    determine (complex) linear characters of $K_0(\CC)$. Since $S$ is non-singular,
    $\{\phi_a\}_{a \in \Pi_\CC}$ is the set of \textit{all} the linear characters of
    $K_0(\CC)$. Observe that $\FPdim$ is a character of $K_0(\CC)$, so that there is
    some $a\in\Pi_\CC$ such that $\FPdim=\phi_a$.  By the unitarity of $S$, we have
    that $\FPdim(\CC)=D^{2}/(d_a)^2$.

    As an abstract group, ${\rm SL}(2,\Z) \cong \langle \fs, \ft \mid \fs^4=1, (\fs \ft)^3
    =\fs^2\rangle$. The standard choice for generators is:
    \begin{equation*}
      \fs := \mtx{0 & -1\\ 1 & 0}\quad \text{and} \quad \ft:=\mtx{1 & 1\\ 0 & 1}\,.
    \end{equation*}

    Let $\eta:\GLC{\Pi_\CC}\to
    \PGLC{\Pi_\CC}$ be the natural surjection. The relations \eqref{eq:STrelations} imply that
    \begin{equation}
      \label{eq:projrep}
      \orho_\CC\colon \fs\mapsto \eta(S)\quad \text{and}\quad \ft\mapsto \eta(T)
    \end{equation}

    defines a projective representation of ${\rm SL}(2,\Z)$.  Since the  modular
    data is an invariant of a modular category, so is the associated projective
    representation type of ${\rm SL}(2,\Z)$.  The following arithmetic properties of this
    projective representation will play an important role in our discussion (cf.
    \cite{NS3}).  Recall that $\BQ_N:=\BQ(\zeta_N)$, where $\zeta_N$ is a primitive $N$th root of unity.
    \begin{thm}
      \label{t:cong1}
      Let $(S,T)$ be the  modular data of the modular category $\CC$ with
      $N=\ord\(T\)$. Then the entries of $S$ are algebraic integers of $\BQ_N$.
      Moreover, $N$ is minimal such that the projective representation $\orho_\CC$ of
      ${\rm SL}(2,\Z)$ associated with the modular data can be factored through $\mathrm{SL}(2,\BZ/N\BZ)$.
      In other words, $\ker\orho_\CC$ is a congruence subgroup of level $N$.
    \end{thm}

    \begin{definition}
      A \defnfont{modular representation} of $\CC$ (cf. \cite{NS3}) is a representation
      $\rho$ of  ${\rm SL}(2,\Z)$ which satisfies the commutative diagram:
      \begin{equation*}
        \xymatrix{ {\rm SL}(2,\Z)\ar[r]^-{\rho}\ar[rd]_-{\orho_\CC}
        &\GLC{\Pi_\CC}\ar[d]^-{\eta}\\
        & \PGLC{\Pi_\CC}\,.
        }
      \end{equation*}
    \end{definition}

    Let $\zeta\in \mbbC$ be a fixed $6$-th root of the anomaly
    $\dfrac{p^{+}}{p^{-}}$. For any $12$-th root of unity $x$, it follows from
    \eqref{eq:STrelations} that  the assignments
    \begin{equation}
      \label{eq:repC}
      \rho_x^\zeta: \fs \mapsto \frac{\zeta^3}{x^3 p^{+}} S, \quad \ft \mapsto \frac{x}{\zeta} T
    \end{equation}

    define a  modular representation of $\CC$.  Moreover, $\{\rho_x^\zeta\mid
    x^{12}=1\}$ is the complete set of modular representations of $\CC$ (cf.
    \cite[Sect. 1.3]{DLN1}).  Since $D^2=p^{+}p^{-}$, we have $\zeta^3/p^{+}=\gamma/D$, where $\gamma =\pm 1$. Thus, one can always find a $6$-th root of unity $x$ so that $\rho_x^\zeta: \fs \mapsto S/D$. For the purpose of this paper, we only need to consider the modular representation $\rho$ of $\CC$ which assigns $\fs \to S/D$.    Note also that $\rho_x^\zeta(\fs)$ and $\rho_x^\zeta(\ft)$
    are matrices over a finite abelian extension of $\BQ$. Therefore, modular
    representations of any modular category are defined over the abelian closure
    $\BQA$ of $\BQ$ in $\BC$ (cf. \cite{BG}).

    Let $\rho$ be any modular representation of the modular category $\CC$, and set
    \begin{equation*}
      s = \rho(\fs) \quad \text{and}\quad t = \rho(\ft)\,.
    \end{equation*}

    It is clear that a representation $\rho$ is uniquely determined by the pair
    $(s,t)$, which will be called a \defnfont{normalized modular pair} of $\CC$. In view of the preceding paragraph, there exists a root of unity $y$ such that $(S/D, T/y)$ is a  normalized modular pair of $\CC$.

  \subsubsection{Galois Symmetry}\label{galsymsection}

    Observe that for any choice of a normalized modular pair $(s,t)$, we have
    $\frac{s_{ia}}{s_{0a}}=\frac{S_{ia}}{S_{0a}}=\phi_a(i)$.
    For each $\s \in \Aut(\BQA)$, $\s(\phi_a)$  given by $\s(\phi_a)(i) =
    \s\left(\frac{s_{ia}}{s_{0a}}\right)$ is again a linear character of $K_0(\CC)$
    and hence $\s(\phi_a) = \phi_{\hs(a)}$ for some unique $\hs\in \Sym(\Pi_\CC)$.
    That is,
    \begin{equation}
      \label{eq:galois1}
      \s\left(\frac{s_{ia}}{s_{0a}}\right) = \frac{s_{i\hs(a)}}{s_{0\hs(a)}} \quad \text{for all }i, a\in
      \Pi_\CC\,.
    \end{equation}
    Moreover, there exists a function $\e_\s : \Pi_\CC \to \{\pm 1\}$, which
    depends on the choice of $s$, such that:
    \begin{equation}
      \label{eq:galois2}
     \s(s_{ij}) = \e_{\s}(i) s_{\hs(i) j} = \e_{\s}(j) s_{i \hs(j)} \quad \text{for all }i, j \in \Pi_\CC
    \end{equation}
    (cf.  \cite[App. B]{BG}, \cite{CG1} or \cite[App.]{ENO1}). The group $\Sym(\Pi_\CC)$ will often be written as $\mathfrak{S}_r$ where $r=|\Pi_\CC|$ is the rank of $\CC$.

    The following theorem will be used in the sequel:
    \begin{thm} \label{t:Galois}
      \label{c:galsym}
      Let $\CC$ be a modular  category of rank $r$, with $T$-matrix of order $N$. Suppose $(s,t)$ is a normalized modular pair of $\CC$. Set $t=(\delta_{ij} t_i)$ and $n =\ord (t)$. Then:
      \begin{enumerate}
       \item[(a)] $N \mid n\mid 12 N$ and
       $s,t \in \GL_r(\BQ_n)$.  Moreover,
       \item[(b)] (Galois Symmetry) for $\s\in \Gal(\BQ_n/\BQ)$,
       $
  \s^2(t_i)=t_{\hs(i)}.
  $
      \end{enumerate}
    \end{thm}
Part $(a)$ of Theorem \ref{t:Galois} is proved in \cite{NS3}, whereas
     part $(b)$ is proved in \cite[Thm. II(iii)]{DLN1}.

In the sequel, we will denote by $\BF_A$
    the field extension over $\BQ$ generated by the entries of a complex matrix
    $A$. If $\BF_A/\BQ$ is Galois, then we write $\Gal(A)$ for the Galois
    group $\Gal(\BF_A/\BQ)$.

In this notation, if $(S,T)$ is the modular data of $\CC$, then $\BF_T = \BQ_N$,
    where $N=\ord\(T\)$, and we have  $\BF_S \subseteq \BF_T$ by \thmref{t:cong1}. In
    particular, $\BF_S$ is an abelian Galois extension over $\BQ$.

For any normalized modular pair $(s,t)$ of $\CC$ we have $\BF_t = \BQ_n$, where
    $n=\ord\(t\)$. Moreover, by  \thmref{t:Galois}, $\BF_S\subseteq \BF_s \subseteq \BF_t$. In
    particular, the field extension $\BF_s/\BQ$ is also Galois. The kernel of the
    restriction map $\res:\Gal(t) \to \Gal(S)$ is isomorphic to $\Gal(\BF_t/\BF_S)$.

    The following important lemma is proved in \cite[Prop. 6.5]{DLN1}.
    \begin{lem}
      \label{l:2group}
      Let $\CC$ be a modular category with modular data $(S,T)$. For any
      normalized modular pair $(s,t)$ of $\CC$, $\Gal(\BF_t/\BF_S)$ is an
      elementary 2-group.
    \end{lem}

  \subsubsection{Frobenius-Schur Indicators}\label{fs indicators}

  Higher Frobenius-Schur indicators are indispensable invariants of spherical categories
    introduced in \cite{NS1}.
    For modular categories the
      Frobenius-Schur indicators can be explicitly computed from the modular data, which we
      take as a definition here:
      \begin{equation}
        \label{eq:BF}
        \nu_n(V_k) = \frac{1}{D^{2}} \sum_{i,j \in \Pi_\CC} N_{i,j}^k\, d_i
        d_j\left(\frac{\theta_i}{\theta_j}\right)^n
      \end{equation}

      for all $k \in \Pi_\CC$ and positive integers $n$.
      There is a minimal $N$ so that $\nu_N(V_k)=d_k$ for all $k\in\Pi_\CC$ called the \textbf{Frobenius-Schur exponent} $\FSexp(\CC)$.  For modular categories we have $\ord\(T\)=\FSexp(\CC)$.

\subsubsection{Modular Data}\label{modulardata}
    \begin{definition}
      Let $S,T\in\GL_r(\BC)$ and define constants $d_j:=S_{0j}$, $\theta_j:=T_{jj}$, $D^2:=\sum_j d_j^2$
      and $p_{\pm}=\sum_{k=0}^{r-1}(S_{0,k})^2\th_{k}^{\pm1}$.  The pair $(S,T)$
      is an \defnfont{admissible modular data} of rank
      $r$ if they satisfy the following conditions:
      \begin{enumerate}
        \item $d_j\in\BR$ and $S=S^t$ with $S\overline{S}^t=D^2 \Id$.
          $T_{i,j}=\delta_{i,j}\theta_i$ with $N:=\ord(T)< \infty$.
        \item $(ST)^3=p^{+}S^2$, $p_{+}p_{-}=D^2$ and $\frac{p_{+}}{p_{-}}$ is a root of unity.
        \item $N_{i,j}^k:=\frac{1}{D^2} \sum_{a=0}^{r-1} \frac{S_{ia} S_{ja}
          \overline{S_{ka}}}{S_{0a}}\in\BN$ for all $0\leq i,j,k\leq (r-1)$.
        \item   $\theta_i\theta_j
          S_{ij}=\sum_{k=0}^{r-1} N_{i^*j}^kd_k\theta_k$ where $i^*$ is the unique label such that $N_{i,i^*}^0=1$.
        \item Define $\nu_n(k): = \frac{1}{D^2} \sum_{i,j =0}^{r-1} N_{i,j}^k\,
          d_i d_j\left(\frac{\theta_i}{\theta_j}\right)^n$. Then
          $\nu_2(k)=0$ if $k\neq k^*$ and $\nu_2(k)=\pm 1$ if $k=k^*$.  Moreover,
          $\nu_n(k)\in \BZ[e^{2 \pi i/N}]$ for all $n,k$.
        \item $\mbbF_S\subset \BF_T=\BQ_N$, $\Gal(\mbbF_S/\BQ)$ is isomorphic to an
          abelian subgroup of $\mathfrak{S}_r$ and $\Gal(\mbbF_T/\mbbF_S)\cong (\BZ/2\BZ)^\ell$ for some integer $\ell$.
        \item (Cauchy Theorem, \cite[Theorem 3.9]{BNRW1}) The prime divisors of $D^{2}$ and $N$ coincide in
        $\BZ[e^{2 \pi i/N}]$.
      \end{enumerate}
    \end{definition}

%

\section{Arithmetic Properties of Modular Categories}
\label{Arithmetic Properties of Modular Categories}

  \subsection{Galois Action on Modular Data}
    \label{Galois Theory}
In this subsection we derive some consequences of the results in Subsection \ref{galsymsection}.

Let $\CC$ be a modular category with admissible modular data $(S,T)$.  The
splitting field of $K_0(\CC)$ is $\BK_\CC=\BQ\left(\frac{S_{ij}}{S_{0j}}\mid
i,j\in\Pi_\CC\right)=\BF_S$, and we define  $\Gal(\CC) =
\Gal(\BK_\CC/\BQ)=\Gal(S)$. We denote by $\BK_j =
\BQ\left(\frac{S_{ij}}{S_{0j}}\mid i\in\Pi_\CC\right)$ for $j \in \Pi_\CC$.
Obviously, $\BK_\CC$ is generated by the subfields $\BK_j$, $j \in \Pi_\CC$.

As in Subsection \ref{galsymsection} there exists a unique $\hs
\in \Sym(\Pi_\CC)$ such that
$$
\s\left(\frac{S_{ij}}{S_{0j}}\right) =\frac{S_{i\hs(j)}}{S_{0\hs(j)}}
$$
for all $i, j \in \Pi_\CC$.  In particular the map
$\s\rightarrow\hs$ defines
an isomorphism between $\Gal(\CC)$ and an (abelian) subgroup of the symmetric group $\Sym(\Pi_\CC)$.
We will often
abuse notation and identify $\Gal(\CC)$ with its image in $\Sym(\Pi_\CC)$, and the $\Gal(\CC)$-orbit of $j \in \Pi_\CC$  is simply denoted by $\langle j \rangle$. Complex conjugation corresponds to the permutation $i \mapsto i\du$ for $i \in \Pi_\CC$. In view of \eqref{eq:orthogonality}, $j \in \Pi_\CC$ is self-dual if, and only if,  $\BK_j$ is real subfield.
\begin{remark}
\label{r:gen_action}
 Since $\BK_\CC$ is Galois over $\BQ$, for any Galois extension $\BA$ over $\BK_\CC$ in $\BC$, the restriction  $\res: \Aut(\BA) \to \Gal(\CC)$ defines a surjective group homomorphism. Therefore, the group $\Aut(\BA)$  acts on $\Pi_\CC$ via the restriction maps onto $\Gal(\CC)$, and so the $\Aut(\BA)$-orbits are the same as the $\Gal(\CC)$-orbits.
 Again, we denote by $\hs$ the associated permutation of  $\s \in \Aut(\BA)$.  Then we have
 $\hs = \id_{\Pi_\CC}$ if, and only if, $\s \in \Gal(\BA/\BK_\CC)$.
\end{remark}

\begin{lem}
\label{l:orbit}
For $j \in \Pi_\CC$ and $\s \in \AQ$, $\BK_{\hs(j)}=\BK_j$.
Moreover, $[\BK_j:\BQ] = |\langle j \rangle| \le |\Pi_\CC|$. If $j$ is self-dual, then every class in the orbit $\langle j \rangle$ is self-dual. In particular, every class in the orbit $\langle 0 \rangle$ is self-dual.
\end{lem}
\begin{proof}
  As we have seen,
$\phi_j: K_0(\CC) \to \BK_\CC$, $\phi_j(i) = \frac{S_{ij}}{S_{0j}}$, defines a
linear character of $K_0(\CC)$. Therefore,
$$
\frac{S_{aj}}{S_{0j}}  \frac{S_{bj}}{S_{0j}} = \sum_{c \in \Pi_\CC} N_{ab}^c
\frac{S_{cj}}{S_{0j}}\,.
$$
Thus, the $\BQ$-linear span of
$
\{ S_{ij}/ S_{0j}\,|\, i \in \Pi_\CC \}
$
is field, and hence equals to $\BK_j$.
Since $\BK_j$ is a subfield of $\BK_\CC$, $\BK_j/\BQ$ is a normal extension.
Therefore,
$$
\BK_j= \s(\BK_j) = \BQ\left(\s\left(\frac{S_{ij}}{S_{0j}}\right)\,\bigg|\, i\in
\Pi_\CC\right) = \BQ\left(\frac{S_{i\hs(j)}}{S_{0\hs(j)}}\,\bigg|\, i\in
\Pi_\CC\right) = \BK_{\hs(j)}\,.
$$
Let $\BA/\BQ$ be any finite Galois extension containing $\BK_j$, and
 $H$  the kernel of the restriction map $\res: \Gal(\BA/\BQ) \to
\Gal(\BK_j/\BQ)$. Then $\s \in H$ if, and only if,
$$
S_{ij}/S_{0j} = \s(S_{ij}/S_{0j}) = S_{i\hs(j)}/S_{0\hs(j)}
$$
for all $i \in \Pi_\CC$. Thus, $H$ is equal to the stabilizer of $j$, and hence
$$
[\BK_j: \BQ] = |\Gal(\BK_j/\BQ)| = |\Aut(\BA)/H| = |\langle j \rangle|.
$$
The last assertion follows immediately from the fact that
$j$ is self-dual if, and only if, $\BK_j$ is a real abelian extension over $\BQ$.
\end{proof}

\begin{lem}
  \label{integrallemma} Let $\CC$ be a modular category with modular data $(S,
T)$.
  \begin{enumerate}
    \item  $\CC$ is pseudo-unitary if and only if $d_i=\pm\FPdim(V_i)$ for
all
    $i \in \Pi_\CC$.
    \item $\CC$ is integral if, and only if, $d_i\in\Z$ for all $i \in
\Pi_\CC$ if, and only if, $|\langle 0\rangle|=1$.
    \item If $|\langle j\rangle |=1$ for all $j\not\in\langle 0\rangle$,
then there exists an $\s\in\Aut\(\bar{\mbbQ}\)$
    such that $(\sigma(S),\sigma(T))$ is realizable modular data for some
    pseudo-unitary modular category.
  \end{enumerate}
\end{lem}

\begin{proof}
The pseudo-unitarity condition is: $\sum_{j \in \Pi_\CC}
d_j^2=\sum_{j \in \Pi_\CC} \FPdim(V_j)^2,$ and $|d_i|\leq\FPdim(V_i)$ so pseudo-unitarity fails if and
only if $|d_i|<\FPdim(V_i)$ for some $i$.  This proves (i).

For (ii), first observe that $|\langle 0\rangle|=1$ if and only if $d_i\in\Z$
for all $i$ proving the second equivalence in (ii).  By \cite[Prop.
8.24]{ENO1} weakly integral fusion categories are pseudo-unitary. Applying
(i) we see that $d_i\in\Z$ if $\FPdim(V_i)\in\Z$. On the other hand, if
$d_i\in\Z$ for all $i$ we have $D^2 = \sum_i d_i^2\in\Z$ and
$\FPdim(V_i)=S_{i,j}/d_j\in\BR$ for some $j$. Since
$\sum_i(S_{i,j})^2=D^2 \in\Z$,
$d_j^2\sum_i(\FPdim(V_i))^2\in\Z$, and in particular
$\FPdim(\CC) \in\BQ$.  But $\FPdim(\CC)$ is an algebraic
integer, so we see that in this case $\CC$ is weakly integral, and hence
pseudo-unitary.

We have $\FPdim(V_i)=S_{i,j}/d_j=\phi_j(i)$ for some $j$.
If $|\langle j\rangle|=1$ then $S_{i,j}/d_j\in\Z$ for $i \in \Pi_\CC$, and so
$\CC$ is
pseudo-unitary. If $|\langle j\rangle|>1$, then $j \in \langle
0\rangle$ by assumption.  Let $\s \in \Aut\(\bar{\mbbQ}\)$ such that
$\hs(0)=j$ (which exists by definition). We consider a Galois conjugate modular
category $\CC'$ with the (realizable) modular data $(\sigma(S),\sigma(T))$. It
is immediate to see that $\CC'$ is pseudo-unitary since $\sigma(\phi_j)$ is the
first row/column of $\s(S)$. This completes the proof of (iii).
\end{proof}

Note that, by \lemmaref{l:orbit}, a modular category which satisfies the condition
(c) of the preceding lemma must be self-dual.

Now we consider a normalized modular pair $(s,t)$ of $\CC$.
 Since $\frac{s_{ij}}{s_{0j}}=\frac{S_{ij}}{S_{0j}}$
we have
$$
\BK_j = \BQ\left(\frac{s_{ij}}{s_{0j}}\mid i \in \Pi_\CC\right)\text{ and }\BK_\CC
=
\BQ\left(\frac{s_{ij}}{s_{0j}}\mid i, j \in \Pi_\CC\right)\,.$$

 For any $\s\in\Aut(\ol\BQ)$, \eqref{eq:galois2} implies that
 \begin{equation}\label{eq:2}
S_{ij} = \e_\s(i)\e_{\s\inv}(j) S_{\hs(i)\hs\inv(j)}\,.
\end{equation}
 for some sign function $\e_\s: \Pi_\CC \to \{\pm 1\}$ depending on $s$.

\begin{rmk}\label{Gidrem}
\mbox{}
\begin{enumerate}
\item
Observe that while $\e_\s(i)$ depends on the choice of the
normalized pair $(s,t)$, the quantity $\e_\s(i)\e_{\s\inv}(j)$ does not.

\item Observe that
$G: \Aut(\ol \BQ) \to GL(\Pi_\CC, \BZ)$, $\s \mapsto \gs:=\s(s)s^{-1}$, defines a group
homomorphism.
If $\gs$ is a diagonal matrix or, equivalently, $\hs = \id_{\Pi_\CC}$, then
$\s(s_{ij}) = \e_\s(j) s_{ij} = \e_\s(i) s_{ij}$ for all $i, j \in \Pi_\CC$. In
particular, $\e_\s(j) s_{0j} = \e_\s(0)
s_{0j}$. Since $s_{0j} \ne 0$ for all $j \in \Pi_\CC$, $\e_\s(j)=\e_\s(0)=\pm 1$
for all $j \in \Pi_\CC$. Therefore, $\gs=\pm I$ if $\hs = \id_{\Pi_\CC}$ (cf. \cite[Lem. 5]{Ban2}). Therefore,  $\im G$ is either isomorphic to $\Gal(\CC)$ or an abelian extension of $\Gal(\CC)$ by $\BZ_2$.
\end{enumerate}
\end{rmk}

The following results  will be useful in
\sectionref{Applications: Rank 5 Classification}.
\begin{lem} \label{l:4.5}
 If $\hs$ is an order 2 permutation in  $\s \in \AQ$, such that $\hs$ has a
fixed point (for example if the rank of $\CC$ is odd) then
$\e_\s(j)=\e_\s(\hs(j))$ and
  \begin{equation*}
    S_{ij} = \e_\s(i)\e_\s(j)S_{\hs(i)\hs(j)}
  \end{equation*}

  for all $i,j \in \Pi_\CC$. In particular,
  \begin{equation*}
    S_{ii} = S_{\hs(i)\hs(i)}
  \end{equation*}

  for all $i\in \Pi_\CC$.
\end{lem}
\begin{proof}
  Let $\ell$ be a fixed point of $\hs$.
  Thus $\s^2(s_{0\ell})=s_{0\ell}$ and so $\e_{\s^2}(\ell)=1$. Since $G_{\s^2}$ is
  diagonal and the $\ell$-th diagonal entry is 1, $G_{\s^2}=\id$ by Remark
\ref{Gidrem}. Thus,
  \begin{equation*}
    s_{0j}=\s^2(s_{0j}) = \e_\s(j)\e_\s(\hs(j)) s_{0j}
  \end{equation*}

  for all $j$. Therefore, $\e_\s(j)=\e_\s(\hs(j))$ for all $j$. On the other hand,
  we always have $\e_\s(j) \e_{\s\inv}(\hs(j))=1$, we find $\e_\s = \e_{\s\inv}$.
  Thus, by \eqref{eq:2}, we have
  \begin{equation*}
    S_{ij}= \e_\s(i)\e_{\s\inv}(j) S_{\hs(i)\hs\inv(j)} = \e_\s(i)\e_{\s}(j)
    S_{\hs(i)\hs(j)}\,.\qedhere
  \end{equation*}
\end{proof}

\begin{lem}
  \label{201304261618}
  If $\mcC$ is a rank $r\geq 5$ modular category with modular data $(S,T)$ such
that
  $\Gal\(\mcC\)=\langle\(0\;1\)\rangle$
  then:
  \begin{enumerate}
    \item $d_1 >0$,
    \item $\frac{1}{d_{1}}+d_{1}$, $D^{2}/d_{1}$, and $d_{i}^{2}/d_{1}$
    are rational integers for $i\geq 2$.
    \item Defining $\e_j=\frac{S_{1j}}{d_j}$ for each $j\geq 2$ we have
\begin{enumerate}
\item $\e_j\in\{\pm 1\}$.
\item There exist $i,j$ such that $\e_i=-\e_j$, and in this case
$S_{ij}=0$.
\end{enumerate}
  \end{enumerate}
\end{lem}
\begin{proof}

By Lemma \ref{l:4.5} we see that $S_{11}=1$. Since $\sigma=\(0\;1\)$ interchanges $d_1$ and $1/d_1$, the trace of $d_1$ is
$d_1+1/d_1$ and the norm of $d_i$ for $i\geq 2$ is $d_i^2/d_1$ so these must be integers.  This implies that
$D^2/d_1=d_1+1/d_1+\sum_{i=2}^{r-1}d_i^2/d_1\in\BZ$.

If the Frobenius-Perron dimension were a multiple of column $j$ for some $j>1$ then $\FPdim(V_i) = S_{ij}/d_j$ is an integer for all $i$ as
 as $|\langle j\rangle|=1$. Then $\CC$ would be integral, and so $d_i \in \BZ$ by Lemma \ref{integrallemma}(ii) for all $i$. However, this contradicts the fact that $|\langle 0 \rangle|=2$.  So the FP-dimension must be a scalar multiple of one of the first two columns. In any of these two cases, we find $d_1>0$.

By \eqref{eq:2}, we have $S_{1j}=\pm S_{0j}$ for $j\geq 2$, so
$\e_j:=\frac{S_{1j}}{d_j}=\pm 1$ proving (iii)(a).  Now orthogonality of the
first two rows of $S$ gives us: $2d_1+\sum_{j\geq 2} \e_jd_j^2=0$ or
$2=-\sum_{j\geq 2}\e_jd_j^2/d_1$, a sum of integers.  Since $r\geq 5$ and $d_j^2/d_1>0$ we see
that it is impossible for all of the $\e_j$ to have identical signs. On the
other hand we have $\e_jd_j=S_{1j}=\e_\s(1)\e_\s(j)S_{0j}$ for each $j\geq 2$,
so $\e_j=\e_\s(1)\e_\s(j)$.  If $\e_i=-\e_j$ then $\e_\s(i)=-\e_\s(j)$ so that
$S_{ij}=\e_\s(i)\e_{\s}(j)S_{\hs(i)\hs(j)}=-S_{ij}$ by Lemma \ref{l:4.5}. Hence $S_{ij}=0$.
\end{proof}

\begin{lem} \label{l:c61}
Suppose $\CC$ is a modular category of \emph{odd} rank $r\geq 5$.
Then  $(0\,1)(2\,\cdots\, r-1) \not\in
\Gal(\CC)$.
\end{lem}
\begin{proof}
  Suppose $\hs=(0\,1)(2\,\cdots\,r-1)$ for some
$\s \in \Gal(\CC)$. Since $S_{ij} = \pm S_{\hs(i) \hs\inv(j)}$ and $r$ odd,
  $$
  S_{11} = \e  \text{ and  } S_{ij} = \e_{ij} S_{02} = \e_{ij} d_2
  $$
 for all  $0 \le i \le 1,\, 2 \le j \le r-1$,  where $\e, \e_{ij}$ are $\pm 1$.
 In particular, the first two rows of the  matrix $S$ are real, $\s(d_2) = \e_{12}
d_2/d_1$, and $\frac{S_{1j}}{S_{0j}}=\frac{\e_{1j}}{\e_{0j}} \in \BZ$ for $j \ge
2$. Thus
  $$
  \frac{\e_{1j}}{\e_{0j}} =\frac{S_{1j}}{S_{0j}} =
\s\left(\frac{S_{1j}}{S_{0j}}\right)=\frac{S_{1\hs(j)}}{S_{0\hs(j)}} \text{ for all } j \ge
2\,,
  $$
  and hence
   $\frac{S_{1j}}{S_{0j}}=\frac{\e_{12}}{\e_{02}}=\e'$ for $j \ge 2$. By
orthogonality of the first two rows of $S$, we find
  $$
  0 = d_1(1+\e)+\sum_{j \ge 2} S_{1j}S_{0j} =d_1(1+\e)+\e'\sum_{j \ge 2}
S_{0j}^2 =d_1(1+\e)+\e'(r-2)d_2^2  \,.
  $$
  Since $r-2 \ne 0 $, $\e=1$  and
  $
  2  = - \e' (r-2) d_2^2/d_1\,.
  $
  Note that both $d_2/d_1$ and $d_2$ are algebraic integers.The equation implies
$d_2^2/d_1 \in \BZ$ and so $(r-2) \mid 2$. This is absurd as $r-2 \ge 3$.
\end{proof}
\begin{lem}\label{l:c62}
Suppose $\CC$ is a modular category of \emph{odd} rank $r\geq 5$. If the
isomorphism classes $r-2, r-1$ are self-dual, then
$$
(0\,1\, \cdots\, r-3)(r-2\,\,r-1) \not\in \Gal(\CC).
$$
\end{lem}
\begin{proof}
 Suppose  $\hs=(0\,1\, \cdots\,
r-3)(r-2\,\,r-1) \in \Gal(\CC)$. Since $S_{ij} = \pm S_{\hs(i)
\hs\inv(j)}$ and $r$ odd,
  $$
  S_{r-1, r-1} = \e  S_{r-2, r-2}  \text{ and  } S_{i,j} = \e_{ij} S_{0, r-2} =
\e_{ij} d_{r-2}
  $$
 for all  $0 \le i \le r-3,\, r-2 \le j \le r-1$,  where $\e, \e_{ij}$ are $\pm
1$. Therefore, for $0 \le i \le r-3$,
 $
\frac{S_{i,r-1}}{d_{r-1}} =  \s\left(\frac{S_{i,r-2}}{d_{r-2}}\right)  =
\frac{S_{i,r-2}}{d_{r-2}}.
 $
 Since the last two columns are real and orthogonal, we find
 $$
 0= S_{r-1, r-2} S_{r-1,r-1}(1+\e) +\sum_{i=0}^{r-3} S_{i,r-2} S_{i,r-1}=
 S_{r-1, r-2} S_{r-1,r-1}(1+\e) +(r-2)d_{r-1}d_{r-2}\,.
 $$
 Since $(r-2)d_{r-1}d_{r-2} \ne 0$ we must have $\e=1$, therefore
 $
 2\frac{S_{r-1,r-2}}{d_{r-2}} \frac{S_{r-1,r-1}}{d_{r-1}} = r-2\,.
 $
  Since $\frac{S_{r-1,r-2}}{d_{r-2}} \frac{S_{r-1,r-1}}{d_{r-1}}$ is an
algebraic
integer, the equation implies it is a rational integer and so  $r-2$ is even, a
contradiction.
 \end{proof}

 For weakly integral modular categories, a positive dimension function is constant on the orbits of the Galois action on
 $\Pi_\CC$ (via $\sigma\rightarrow \hs$):

\begin{lemma}
  \label{201301301500}

  Let $\CC$ be a weakly integral modular in which $d_a> 0$ for all $a \in \Pi_\CC$.  Then
 we have $d_{\hs\(a\)}=d_{a}$ for all $\s\in\Gal\(\mcC\)$ and $a \in \Pi_\CC$.

\end{lemma}
\begin{proof}

Since $\CC$ is weakly integral, $d_a^2/D^2 \in \BQ$.
Consider the Galois group action on the normalized $S$-matrix $s=\frac{1}{D}S$. We find
$d_a^2/D^2 = \s(d^2_a/D^2) =  d^2_{\hs(a)}/D^2$ for all $\s\in\Gal\(\mcC\)$ and $a \in \Pi_\CC$, and so the result follows.
\end{proof}

  \subsection{Modularly Admissible Fields}
    \label{Number Theory}

The abelian number fields $\mbbF_t,\mbbF_T,\mbbF_{s}$ and $\mbbF_S$ described in Section \ref{galsymsection} (see also \cite{DLN1,RSW}) have the lattice relations
\begin{align}
\label{201303031822}
\vcenter{\hbox{
\begin{tikzpicture}[scale=0.75]
  \node[] (t) at (0,0) {$\mbbF_{t}$};
  \node[] (T) at (-2,-2) {$\mbbF_{T}$};
  \node[] (s) at (2,-2) {$\mbbF_{s}$};
  \node[] (S) at (0,-4) {$\mbbF_{S}$};
  \path[-,draw=black] (t) -- (T);
  \path[-,draw=black] (t) -- (s);
  \path[-,draw=black] (s) -- (S);
  \path[-,draw=black] (T) -- (S);
\end{tikzpicture}
}}\,.
\end{align}
Moreover, by Lemma \ref{l:2group}, the Galois group $\Gal(\BF_t/\BF_S)$ is an elementary 2-group. This implies all the subextensions among these fields will satisfy the same condition. We will call the extension $\BL/\BK$ \defnfont{modularly admissible} if $\BL$ is a cyclotomic field and $\Gal(\BL/\BK)$ is an elementary 2-group, i.e. $\BL$ is a multi-quadratic extension of $\BK$.
In this section we will describe the conductor of a cyclotomic field $\BL$ when $\BL/\BK$ is modularly admissible and $[\BK: \BQ]$ is a prime power.

\begin{remark}
If $\BL/\BK$ is modularly admissible, then $\BL'/\BK'$ is also modularly admissible for any subextensions $\BK'\subset \BL'$ of $\BK$ in $\BL$.  In particular, $\BQ_\mff/\BK$ is modularly admissible where $\mff:=\mff\(\mbbK\)$ is the \emph{conductor} of
$\mbbK$, i.e. the smallest integer $n$ such that $\BK$ embeds into $\BQ_n$.
\end{remark}

A restatement of \cite[Prop. 6.5]{DLN1} in this terminology is:
\begin{prop}
  \label{201302041108}
  If $\BQ_n/\BK$ is modularly admissible and  $\mff$ is the conductor of $\BK$, then
  \begin{enumerate}
    \item $\frac{n}{\mff} \divides 24$  and $\gcd(\frac{n}{\mff} ,\mff)\divides 2$ and
    \item $\Gal\(\BQ_n/\mbbQ_{\mff}\)$ is subgroup of $(\BZ/2\BZ)^3$. \hfill\qed
  \end{enumerate}
\end{prop}

Recall that $\Gal(\BQ_n/\BQ) \cong (\BZ/n \BZ)^\times$ and that for any subfield $\BK$ of $\BQ_n$, we have the exact sequence
\begin{equation}\label{eq:exact_sequence}
1 \to \Gal(\BQ_n/\BK) \to \Gal(\BQ_n/\BQ) \xrightarrow{\res} \Gal(\BK/\BQ) \to 1\,.
\end{equation}

In addition, if $\BQ_n/\BK$ is modularly admissible, then $\Gal(\BQ_n/\BK)$ is isomorphic to a subgroup of the maximal elementary 2-subgroup, $\Omega_2(\BZ/n \BZ)^\times$, of $(\BZ/n \BZ)^\times$. In particular, $\frac{(\BZ/n \BZ)^\times}{\Omega_2(\BZ/n \BZ)^\times}$ is a homomorphic image of $\Gal(\BK/\BQ)$.

\begin{lem}\label{p:odd_primes}
  If $\BQ_n/\BK$ is modularly admissible and $\[\BK:\mbbQ\]$ is
  odd, then $\Gal(\BK/\BQ) \cong \frac{(\BZ/n \BZ)^\times}{\Omega_2(\BZ/n \BZ)^\times}$ and $q \equiv 3 \mod 4$ for any odd prime $q \mid n$. If, in addition, $\[\BK:\mbbQ\]$ is a power of an odd prime $p$, then every prime factor $q>3$ of $n$ is a simple factor of the form $q = 2p^r+1$ for some  integer $r\ge 1$. Moreover, if $p > 3$, then $r$ must be odd and $p \equiv 2 \mod 3$.
\end{lem}
\begin{proof} It follows from the exact sequence \ref{eq:exact_sequence} that
$\Gal(\BQ_n/\BK)$ is a Sylow 2-subgroup of $\Gal(\BQ_n/\BQ)$ and hence
$\Gal(\BQ_n/\BK) = \Omega_2 \Gal(\BQ_n/\BQ)$. Therefore, we obtain the isomorphism $\Gal(\BK/\BQ) \cong \frac{(\BZ/n \BZ)^\times}{\Omega_2(\BZ/n \BZ)^\times}$. Suppose $q>3$ is a prime factor of $n$ and $\ell$ is the largest integer such that $q^\ell \mid n$. Then, by the Chinese Remainder Theorem  $(\BZ/{q^{\ell}}\BZ)^\times$ is a direct summand of $\(\mbbZ/n\mbbZ\)^{\times}$, and hence $\frac{(\BZ/{q^{\ell}}\BZ)^\times}{\Omega_2(\BZ/{q^{\ell}}\BZ)^\times}$ is isomorphic to a subgroup of $\Gal(\BK/\BQ)$. In particular, $\vph(q^\ell)/2 = q^\ell \(\frac{q-1}{2}\)$ is odd, and this implies $q \equiv 3 \mod 4$.

If, in addition, $[\BK:\BQ]=p^h$ for some $h \ge 0$, then $q^{\ell-1}\(\frac{q-1}{2}\) \mid p^h$ when $q>3$. This forces $\ell=1$ and $q =2\cdot p^r +1$ for some positive integer $r \le h$. Furthermore, if $p> 3$,  then $q =2\cdot  p^r +1 \equiv 0\mod 3$ whenever $r$ is even or $p \equiv 1\mod 3$. The last statement then follows.
\end{proof}

When the abelian number field $\BK$ has a prime power degree over $\BQ$, more refined statements on a modularly admissible extension $\BQ_n/\BK$ can now be stated as

\begin{prop}
  \label{201301221709}
  Let $\BQ_n/\BK$ be a modularly admissible extension and $$\Gal(\BK/\BQ)\cong \BZ/p^{r_1} \BZ \times \cdots \times \BZ/p^{r_m} \BZ$$  for some prime $p$ and  $0<r_1 \le \cdots \le r_m$, and set $q_j= 2\cdot p^{r_j}+1$ for $j=1,\ldots,m$. Then:
  \begin{enumerate}
    \item If $p>3$, then
    $n$ admits the factorization
    $
    n=f \cdot q_1\cdots q_m
    $
    where $f \mid 24$ and $q_1 , \dots, q_m$ are distinct primes. In particular, $r_1, \ldots, r_m$ are distinct odd integers and $p \equiv 2 \mod 3$.
    \item For $p=3$,  one of the following two statements holds.
    \begin{enumerate}
    \item $9 \nmid n$ and
    $
     n = f \cdot q_1\cdots q_m
    $
    where $f \mid 24$ and $q_1 , \dots, q_m$ are distinct primes.
   \item $9 \mid n$, and there exists $i \in \{1, \dots, m\}$ such that
   $\{q_j \mid j \ne i\}$ is a set of $m-1$ distinct primes and
    $
    n =  f \cdot 3^{r_i+1}\cdot q_1\cdots q_{r_{i-1}}\cdot q_{r_{i+1}}\cdots q_m
    $
   where $f \mid 8$.
    \end{enumerate}
    \item For $p=2$,
    $
    n=2^a \cdot p_1 \cdots p_l
    $
    where $p_1, \dots, p_l$ are distinct Fermat primes and $a$ is a non-negative integer.
  \end{enumerate}
\end{prop}
\begin{proof} For $p=2$, the exact sequence \ref{eq:exact_sequence} implies that $\Gal(\BQ_n/\BQ)$ is a 2-group and so $\vph(n)$ is a power of $2$. Hence,  (iii) follows.

For  any odd prime $p$, it follows from Lemma \ref{p:odd_primes} that $n=2^a 3^b q_1 \cdots q_l$ for some integers $a, b \ge 0$ and odd primes $q_1 < \cdots < q_l$ of the form $q_j = 2 p^ {a_j}+1$ for some integer $a_j \ge 1$. Therefore,
\begin{equation}
\Gal(\BK/\BQ)  \cong \frac{(\BZ/n\BZ)^\times}{\Omega_2(\BZ/n\BZ)^\times } \cong  \frac{(\BZ/2^a\BZ)^\times}{\Omega_2(\BZ/2^a\BZ)^\times } \times \frac{(\BZ/3^b\BZ)^\times}{\Omega_2(\BZ/3^b\BZ)^\times } \times
\BZ/p^{a_1} \BZ \times \cdots \times \BZ/p^{a_l}\BZ \,.
\end{equation}
For $p > 3$, $(\BZ/2^a\BZ)^\times \times (\BZ/3^b\BZ)^\times$ must be an elementary 2-group otherwise the $p$ power $\left|\frac{(\BZ/n \BZ)^\times}{\Omega_2(\BZ/n \BZ)^\times}\right|$ has a factor of $2$ or $3$. Therefore,
$0 \le a \le 3$, $0 \le b \le 1$  (or equivalently, $f = 2^a 3^b$ is a divisor of 24), and
$$
\Gal(\BK/\BQ)  \cong \BZ/p^{a_1} \BZ \times \cdots \times \BZ/p^{a_l}\BZ\,.
$$
By the uniqueness of invariant factors, $l = m$ and $a_j = r_j$ for $j=1, \dots, m$.  The last statement of (i) follows directly from Lemma \ref{p:odd_primes}.

For $p=3$ and $9 \nmid n$,  the argument for the case $p>3$ can be repeated here to arrive the same conclusion (iii)(a). For $p=3$ and $9 \mid n$,
$b \le 2$ and so
  $$
\Gal(\BK/\BQ)   \cong  \frac{(\BZ/2^a\BZ)^\times}{\Omega_2(\BZ/2^a\BZ)^\times } \times \BZ/3^{b-1}\BZ
\times
\BZ/p^{a_1} \BZ \times \cdots \times \BZ/p^{a_l}\BZ \,.
$$
 Therefore, $ (\BZ/2^a\BZ)^\times$ is an elementary 2-group, or $0 \le a \le 3$. By the uniqueness of invariant factors $l =m-1$, $b-1 = r_i$ for some $i$ and
$(a_1, \dots, a_{m-1}) = (r_1, \dots, \hat{r_i}, \dots r_m)$. This proves (iii)(b).
\end{proof}

\begin{cor}
  \label{201302011406}
If $\BQ_n/\BK$ is modularly admissible and $\BK/\BQ$ is a multi-quadratic extension, then $n \mid 240$.
\end{cor}
\begin{proof}
Since $\Gal(\BQ_n/\BK)$ and $\Gal(\BK/\BQ)$ are elementary 2-groups, in view of \eqref{eq:exact_sequence}, $\Gal(\BQ_n/\BQ)$ is an abelian 2-group whose exponent $e \mid 4$. By Proposition \ref{201301221709} (iii), $n=2^a p_1\cdots p_l$ where $a \ge 0$ and $p_1 < \cdots< p_l$ are  Fermat primes. If $p_l > 5$, then $\Gal(\BQ_n/\BQ)$ has a cyclic subgroup of order $p_l-1>4$; this contradicts $e\mid 4$. On the other hand, if $a \ge 5$, $\Gal(\BQ_n/\BQ)$ has a cyclic subgroup of order 8 which is also absurd. Therefore, $n$ must be a factor of $2^4 \cdot 3\cdot 5=240$.
\end{proof}

These techniques combined with the Cauchy Theorem  \cite[Theorem 3.9]{BNRW1}
can be used to classify low rank integral modular categories with a given Galois group.  For example:
\begin{lemma}
    There are no rank 7 integral modular categories satisfying $\Gal\(\mcC\)\cong\mbbZ/5\mbbZ$.
  \end{lemma}
 \begin{proof} We may assume $d_a > 0$ for all $a \in \Pi_\CC$, by Lemma \ref{integrallemma}.
 By applying Lemma \ref{201301301500} we see that the dimensions are $1$, $d_1$ (with multiplicity $5$) and $d_2$ (with multiplicity $1$).
In this case \propref{201301221709}(i) and the Cauchy Theorem imply that the
prime divisors of $d_1,d_2$ and $D^2$ lie in $\{2,3,11\}$. Moreover, $D\in\BZ$ since $|\Gal(\CC)|$ is odd (recall that $\Gal(\CC)=\Gal(S)$).  Examining the dimension equation $D^2=1+d_1^2+5d_2^2$ modulo $5$ we obtain $D^2=1+d_1^2$.  The non-zero squares modulo $5$ are $\pm 1$, so $D^2,d_1^2\in\{\pm 1\}$ which give no solutions.
 \end{proof}

  \subsection{Representation Theory of ${\rm SL}(2,\Z)$}
    \label{Representation Theory of SL(2,Z)}

\begin{defn}
  Let $\rho: {\rm SL}(2,\Z) \to \GLC{r}$ be a representation of ${\rm SL}(2,\Z)$.
  \begin{enumerate}
    \item[(i)] $\rho$ is said to be \defnfont{non-degenerate} if the $r$ eigenvalues
      of $\rho(\ft)$ are distinct.
    \item[(ii)] $\rho$ is said to be \defnfont{admissible} if there exists modular
      category $\CC$ over $\BC$ of rank $r$ such that $\rho$ is a modular
      representation of $\CC$ relative to certain ordering of $\Pi_\CC = \{V_0, V_1,
      \dots, V_{r-1}\}$ with $V_0$ the unit object of $\CC$. In this case, we say that
      $\rho$ can be \defnfont{realized} by the modular category $\CC$.
    \item[(iii)] $\ARep({\rm SL}(2,\Z))$ denotes the set of all complex admissible
      ${\rm SL}(2,\Z)$-representation.
  \end{enumerate}
\end{defn}


By \cite{DLN1}, an admissible representation $\rho: {\rm SL}(2,\Z) \to \GLC{r}$  must
factor through $\SL{n}$ where $n=\ord\rho(\ft)$, and $\rho$ is $\BQ_n$-rational.  In particular
admissible representations are completely reducible.
It follows from \cite[Lem. 1]{Eh1} that each non-degenerate admissible
representation of ${\rm SL}(2,\Z)$ is absolutely irreducible.  Moreover, by \cite{TubaWenzl} any irreducible
representation of ${\rm SL}(2,\Z)$ of dimension at most $5$ must be non-degenerate.

\begin{lem}
  \label{nondeglem}
  Let $\rho$ be a degree $r$ non-degenerate admissible representation of ${\rm SL}(2,\Z)$
  with $t=\rho(\ft)$ and $s=\rho(\fs)$. Suppose $\rho' \in \ARep({\rm SL}(2,\Z))$ is
  equivalent to $\rho$  with $t'=\rho'(\ft)$ and $s'=\rho'(\fs)$. Then
  $\rho'(\fg)=U\inv\rho(\fg)U$ for a signed permutation matrix $U\in \GLC{r}$ of
  the permutation $\vs$ on $\{0, \dots, r-1\}$ defined by $t'_{\vs(i)}=t_i\,.$

  If, in addition, $t_0=t'_0$, then $\vs$ defines an isomorphism of fusion rules
  associated to $\rho$ and $\rho'$.
\end{lem}
\begin{proof}
  Since $\rho$ and $\rho'$ are equivalent, $t$ and $t'$ have the same
  eigenvalues. By the non-degeneracy of $\rho$, there exists a unique permutation
  $\vs$ on $\{0,\dots, r-1\}$ defined by  $t'_{\vs(i)} = t_i$. Let
  $D_\vs=[\delta_{\vs(i)j}]_{i,j}$ be the permutation matrix of $\vs$. Then
  $\rho''=D_\vs \rho' D_\vs\inv$ is equivalent to $\rho$ and $\rho''(\ft)=t$.
  There exists $Q \in \GLC{r}$ such that $Q\rho''  = \rho Q$. Since $Qt  = t Q$
  and $t$ has distinct eigenvalues, $Q$ is a diagonal matrix. Suppose
  $Q=[\delta_{ij}Q_i]_{i,j \in \Pi_\CC}$. Then
  \begin{equation*}
    s'_{\vs(i)\vs(j)} = \frac{Q_j}{Q_i} s_{ij}\,.
  \end{equation*}

  Both $s$ and $s'$ are symmetric, and so we have
  \begin{equation*}
    \frac{Q_0}{Q_j} s_{0j} = \frac{Q_j}{Q_0} s_{j0}=\frac{Q_j}{Q_0} s_{0j}\,.
  \end{equation*}

  Since $s_{0j} \ne 0$, $\frac{Q_j}{Q_0}=\pm 1$. Let $Q'=\frac{1}{Q_0} Q$ and
  $U=Q'D_\vs$. Then ${Q'}^2=I$, $U$ is a signed permutation matrix of $\vs$, and
  $s'= U\inv s U$. Since there are finitely many signed permutation matrices in
  $\GLC{r}$, the equivalence class of admissible representations of $\rho$ is
  finite.

  If, in addition, $t_0=t'_0$, then $\vs(0)=0$. Let $(s')\inv=[\ol
  s'_{i'j'}]_{i',j' \in \Pi_{\CC'}}$ and $s\inv=[\ol s_{ij}]_{i,,j \in
  \Pi_{\CC}}$.
  By the Verlinde formula,
  \begin{equation*}
    N_{\vs(i)\vs(j)}^{\vs(k)} =
    \sum_{a=0}^{r-1}\frac{s'_{\vs(i)\vs(a)}s'_{\vs(j)\vs(a)} \ol
    s'_{\vs(k)\vs(a)}}{s'_{0\vs(a)}} =Q'_i Q'_j Q'_k
    \sum_{a=0}^{r-1}\frac{s_{ia}s_{ja} \ol s_{ka}}{s_{0a}} = Q'_i Q'_j Q'_k
    N_{ij}^k\,.
  \end{equation*}

  Thus, $Q'_i Q'_j Q'_k=1$ whenever $N_{ij}^k\ne 0$. Moreover, $\vs$ defines an
  isomorphism between the fusion rules of $\CC$ and $\CC'$.
\end{proof}

Let  $\rho: {\rm SL}(2,\Z) \to \GLC{n}$ a representation.
The set of eigenvalues of $\rho(\ft)$ is called the $\ft$-\emph{spectrum} of
$\rho$.  The minimal $N>1$ such that $\rho$ factors over ${\rm SL}(2,\Z/N\Z)$ is called the \textbf{level} of $\rho$.
\begin{lem}
  \label{l:1}

  Let $\CC$ be a modular category of rank $r$,  and $\rho: {\rm SL}(2,\Z) \to \GLC{r}$ a modular representation of $\CC$. Then $\rho$ cannot be isomorphic to a direct sum of two representations with disjoint $\ft$-spectra. In particular, if $\rho$ is non-degenerate, then it is absolutely irreducible.
\end{lem}
\begin{proof}
  Let $s=\rho(\fs)$ and $t=\rho(\ft)$. Then $\frac{s_{0j}}{s_{00}}$ is the
  quantum dimension of the simple object $j$. In particular, every entry of the
  first row of $s$ is non-zero. Thus, for any permutation matrix $Q$, there exists
  a row of $Q\inv s Q$ which has no zero entry.

  Suppose $\rho$ is isomorphic a direct sum of two matrix representations
  $\rho_1, \rho_2$ of ${\rm SL}(2,\Z)$ with disjoint $\ft$-spectra. Since $\rho(\ft)$ has
  finite order, and so are $\rho_i(\ft)$, $i=1,2$. Without loss of generality, we
  can assume $\rho_1(\ft)$ and $\rho_2(\ft)$ are diagonal matrices. There exists a
  permutation matrix $Q$ such that  $Q\inv t Q = \bmtx{\rho_1(\ft) & 0 \\\hline 0
  & \rho_2(\ft)}$. Since the representation $\rho^Q: {\rm SL}(2,\Z) \to \GLC{k}$, $\fs
  \mapsto Q\inv sQ, \ft \mapsto Q\inv t Q$ is also equivalent to $\rho_1\oplus
  \rho_2$, there exists $P \in \GLC{k}$ such that
  \begin{equation*}
    P\bmtx{\rho_1(\ft) & 0 \\ \hline 0 & \rho_2(\ft)} =  \bmtx{\rho_1(\ft) & 0
    \\\hline 0 & \rho_2(\ft)} P \quad\text{ and } \quad Q\inv s Q =
    P\bmtx{\rho_1(\fs) & 0 \\ \hline 0 & \rho_2(\fs)} P\,.
  \end{equation*}

  Since $\rho_1$ and $\rho_2$ have disjoint $\ft$-spectra, $P$ must be of the
  block form
  $\bmtx{P_1 & 0 \\\hline 0 & P_2}$. This implies every row of $Q\inv sQ$ has at
  least one zero entry, a contradiction.
\end{proof}

\begin{cor}
  \label{c:1}
  Suppose $\CC$ is a modular category of rank $r>2$, and $\rho$ is a modular
  representation of $\CC$. Then:
  \begin{enumerate}
  \item $\rho$ cannot be a direct sum of 1-dimensional representations of ${\rm SL}(2,\Z)$.
  \item If $\rho_1$ is a subrepresentation of  degree $r-2$,
  then the $\ft$-spectrum of $\rho_1$ must contain a $120$-th root of unity.
\end{enumerate}
\end{cor}
\begin{proof}
  The statement (i) was proved in \cite{Eh1} using a simpler version of Lemma \ref{l:1}.

  Suppose $\rho_1$ is a degree $r-2$ subrepresentation of $\rho$ such that
  $\w^{120}\ne 1$ for all eigenvalues $\w$ of $\rho_1(\ft)$. Then there exists a
  2-dimensional representation $\rho_2$ of ${\rm SL}(2,\Z)$ such that $\rho\cong
  \rho_1\oplus \rho_2$.

  If $\rho_2$ is a sum of 1-dimensional subrepresentations, then $\rho_2(\ft)^{12}
  = \id$. If $\rho_2$ is irreducible, then applying the Chinese remainder theorem we obtain $\rho_2 \cong \xi \o \phi$ for some
  linear character $\phi$, and an irreducible representation  $\xi$ of prime power
  level. It follows from \cite[Table A1]{Eh2} that $\rho_2(\ft)^{120}=\id$.
  Thus,  for both cases, $\rho_1$ and $\rho_2$ have disjoint $\ft$-spectra.
  However, this contradicts Lemma \ref{l:1}.
\end{proof}

For any representation $\rho$ of ${\rm SL}(2,\Z)$, we say that $\rho$ is \textbf{even}
(resp. \textbf{odd}) if $\rho(\fs)^2=\id$ (resp. $\rho(\fs)^2=-\id$).  We denote the set of primitive $q$-th roots of unity by $\mu_q$, the set of all $q$-th roots of unity by $\ol \mu_q$,
 and  $\mu_{q*} = \bigcup\limits_{n \in \BN} \mu_{qn}$.

\begin{remark}
  \label{r1} If $\rho$ is even, then the linear representation $\det \rho$ of ${\rm SL}(2,\Z)$ is also even, and so $\det\rho(\ft) \in \ol\mu_6$.  In general,
a representation of ${\rm SL}(2,\Z)$ may neither even nor odd. However, if $\CC$ is a
self-dual modular category, then $\CC$ admits an \emph{even} modular
representation given by the normalized modular pair $(\frac{1}{D} S, \frac{1}{\zeta} T)$  for any 3-rd root $\zeta$ of $\frac{D}{p^-}$.
  Let $\rho$ be a modular representation of $\CC$. Then for any
  linear character $\chi$ of ${\rm SL}(2,\Z)$, there exists a modular representation $\rho'
  \cong \rho \o\chi$ as representations of ${\rm SL}(2,\Z)$. In addition, if $\rho$ and
  $\chi$ are even, then so is $\rho'$.
\end{remark}

\begin{lem}\label{l:4.25}
  Suppose $\CC$ is a self-dual modular category of rank $r$, and $\rho$ is an
  even modular representation of $\CC$. If  $\rho \cong \phi_1 \oplus (\phi_2\o
  \xi)$ for some degree 1 representations $\phi_1$, $\phi_2$ and a degree $r-1$
  non-degenerate irreducible representation  $\xi$ of ${\rm SL}(2,\Z)$ with odd level, then
  $\phi_1$, $\phi_2$ and $\xi$ are all even.
\end{lem}
\begin{proof}
  Let $\w_1 = \phi_1(\ft)$ and $\w_2=\phi_2(\ft)$. Note that $\phi_i(\fs) =
  \w_i^{-3}$ for all $i=1,2$, and $\w_i^{12}=1$. Since $\rho$ is even,  $\phi_1$
  and $\phi_2 \o \xi$ are even. In particular, $\w_1^6=1$.
  Since $\rho$ is reducible and $\xi$ is non-degenerate, by Lemma \ref{l:1},
  $\w_1 \w_2\inv$ must be in the spectrum of $\xi(\ft)$ . Therefore, $\w_1
  \w_2\inv$ is of odd order, and hence $\w_2^6=1$. Therefore,  $\phi_2$ is  even.
  Since $\phi_2\o \xi$ is even, $\xi$ is also even.
\end{proof}

\begin{remark}
\label{l:2}
  If $\rho$ is a modular representation of a modular category  $\CC$, then the order of its $T$-matrix  is equal to the \emph{projective order} of $\rho(\ft)$, i.e. the small positive integer $N$ such that $\rho(\ft)^N$ is a scalar multiple of the identity.
\end{remark}

\begin{lem} \label{l:SUequivalence}
Let $\CC$ be a fusion category such that $G(\CC)$ is trivial and $\KK_0(\CC) \o_\BZ \BZ_N$ is isomorphic to $\KK_0(SU(N)_k)$ for some integer $k$ relatively prime to $N$. Then $\CC$ is monoidally equivalent to a Galois conjugate of $SU(N)_{k}/\BZ_N$.\footnote{See Section \ref{Applications: Rank 5 Classification} for notation.}
\end{lem}
\begin{proof}
 Let $\mcS$ be a rank $N$ fusion category with fusion rules $\BZ_N$ (or $\Vec(\BZ_N)$).
 Now, we have
 $$
 \KK_0(\CC \boxtimes \mcS) \cong \KK_0(\CC) \ot \KK_0(\mcS) \cong \KK_0(SU(N)_k)
 $$
 as based rings.
 By the classification in \cite{KazWen}, $\CC \boxtimes \mcS$ is monoidally equivalent to $\DD \boxtimes \Vec(\BZ_N, \w)$ for some $3$-cocycle $\w$ of $\BZ_N$ and  Galois conjugate $\DD$ of $SU(N)_k/\BZ_N$ (\emph{i.e.} a choice of a root of unity).   As these categories are $\BZ_N$-graded and the adjoint subcategories ($\CC$ and $\DD$ respectively) are the $0$-graded components we have that $\CC$ is monoidally equivalent to $\DD$.
\end{proof}
Recall from Section \ref{Galois Theory} that $\BK_0$ is the extension of $\BQ$ generated by the dimensions of the simple objects in $\CC$.
\begin{thm}
  \label{t:irr}
  Let $\CC$ be a modular category such that
  $|\Pi_\CC|=[\BK_0:\BQ]=p$ is a prime.  Then:
\begin{enumerate}
\item  Every modular representation of $\CC$ is non-degenerate and hence absolutely irreducible.
\item $q=2p+1$ is a prime.
\item $\FSexp(\CC)=q$.
\item The underlying fusion category of $\CC$ is monoidally equivalent to a Galois conjugate of $SU(2)_{2p-1}/\BZ_2$.
\end{enumerate}
\end{thm}
\begin{proof}
The cases $p=2,3$ follow from the classification in \cite[pp. 375--377]{RSW}.  We may assume $p>3$.

Let $\rho$ be a modular representation of $\CC$, and set $s=\rho(\fs)$, $t=\rho(\ft)$ and $n=\ord(t)$.
  By Lemma \ref{l:orbit}, $|\langle  0 \rangle|= [\BK_0:\BQ] = |\Pi_\CC|$. Thus,
  $\BK_\CC=\BK_0$ and so $|\Gal(\CC)| = p$. Thus, $\Gal(\CC)\cong
  \BZ_p$.  Let $\s \in \GalQ{n}$ such that
  $\s|_{\BK_\CC}$ is a generator of $\Gal(\CC)$, and hence $\hs = (0, \hs(0),
  \hs^2(0), \dots, \hs^{p-1}(0))$.
  By Theorem \ref{t:Galois}, $t_{\hs^i(0)}= \s^{2i}(t_0)$. Thus, $\BQ_n = \BQ(t_0)$.  Suppose $t_{\hs^i(0)} = t_{\hs^j(0)}$ for
  some
  non-negative integers $ i < j \le p-1$. Then $\s^{2(j-i)}(t_0)=t_0$ and so $\s^{2(j-i)}=\id$. This implies $\hs^{2l}=\id$ for some positive integer $l \le p-1$, and hence $p \mid 2 l$, a contradiction. Therefore, $t_{\hs^i(0)} \ne t_{\hs^j(0)}$ for
  all non-negative integers $ i < j \le p-1$, and hence $\rho$ is non-degenerate. By Lemma \ref{l:1}, $\rho$ is absolutely irreducible.

  Note that $(\BF_S, \BF_t)$ is a modularly admissible, and $\BF_S=\BK_\CC$ and $\BF_t = \BQ_n$.
  Since
  $[\BK_0: \BQ]=|\langle 0 \rangle| = |\Pi_\CC|$, $\BK_0 = \BK_\CC$.
  By Proposition \ref{201301221709} (since $p>3$) we have $q=2p+1$ is a prime and $q \mid n \mid 24 q$.

Since $(q,24)=1$, by the Chinese Remainder Theorem, $\rho \cong \chi \o R$ for some irreducible representations $\chi$ and $R$ of levels $n/q$ and $q$ respectively. Since $q \mid n$ and $12 \nmid q$, $R$ is not linear. Thus, the prime degree $p$ of $\rho$ implies that $\deg R = p$ and $\deg \chi =1$.  Since $\rho(\ft)^q = \chi(\ft)^q \o \id$,  $\FSexp(\CC) \mid q$ by Remark \ref{l:2}, and hence $\FSexp(\CC)=q$.

Since $\FSexp(\CC)=q$ is odd,  there exists a modular representation $\rho$ of $\CC$ with level $q$ by \cite[Lem. 2.2]{DLN1}.
There is a dual pair of such irreducible representations of $\SL{q}$.
Realizations can be obtained from
  the modular data for $\DD=SU(2)_{2p-1}/\BZ_2$ (see e.g. \cite{BKi}):
  \begin{equation}
    S_{i,j}=\frac{\sin\left(\frac{(2i+1)(2j+1)\pi}{q}\right)}{\sin\left(\frac{\pi}{
    q}\right)}, \quad \theta_j=e^{\frac{2\pi i(j^2+j)}{q}}
  \end{equation}

  where $0\leq j\leq (p-1)=\frac{q-3}{2}$.  Since the $\theta_j$ are distinct and the
  $T$-matrix has order $q$, we can normalize $(S_\DD,T_\DD)$ to a pseudo-unitary
  modular pair $(\tilde{s},\tilde{t})$ corresponding to a degree $p$ and level $q$ irreducible representation of ${\rm SL}(2,\Z)$. Complex conjugation gives the other
  inequivalent such representation, and both have the first column a multiple of the Frobenius-Perron dimension.

  By Lemma \ref{integrallemma}(iii) we may replace the modular data $(S_\CC,T_\CC)$
  by an admissible pseudo-unitary modular data $(S^\prime,T^\prime)$.  After normalizing
  and taking the complex conjugates (if necessary) we can assume that the resulting pair
  $(s^\prime,t^\prime)$ is conjugated to $(\tilde{s},\tilde{t})$ by a signed permutation $\vs$,
  by Lemma \ref{nondeglem}. The first row/column of both
  $s^\prime$ and $\tilde{s}$ are projectively positive.  The first
  column of $s^\prime$ is mapped to the first column of $\tilde{s}$ under $\vs$.  In
  particular $\vs$ fixes the label $0$ (as the Frobenius-Perron dimension is the unique projectively positive column of any $S$-matrix) so the last part of Lemma \ref{nondeglem}
  implies that the fusion rules coincide.  Now, statement (iv) follows from Lemma \ref{l:SUequivalence} as there are exactly $N$ invertible objects in $SU(N)_k$, labeled by weights at the corners of the Weyl alcove.
\end{proof}

\begin{lem}
  \label{l:inadp}
  Let $p>3$ be a prime. Then the unique degree $p$ irreducible representation $\psi$
  of $\SL{p}$ is not admissible.
\end{lem}
\begin{proof}
  The result was established in \cite{Eh2} by using the integrality
of
  fusion rules and Verlinde formula. Here we provide another proof by using the
  rationality of modular representations of any modular category. Suppose there
  exists a modular category $\CC$ of rank $p$ which admits a modular representation $\rho$ equivalent to $\psi$ as representations of ${\rm SL}(2,\Z)$. The
  representation $\psi$ is given by
  \begin{eqnarray*}
    \psi(\ft)_{jk} & = & \delta_{jk} e^{\frac{2\pi i k}{p}} \\
    \psi(\fs)_{00} & = & \frac{-1}{p}\\
    \psi(\fs)_{0k} =  \psi(\fs)_{k0}& = & \frac{\sqrt{p+1}}{p}\quad \text{for } 0
    < k < p,\\
    \psi(\fs)_{jk} &=& \frac{1}{p} \sum_{a=1}^{p-1} e^{\frac{2\pi i}{p} (a j +
    a\inv k)} \quad \text{for } 0 < j,k < p\,.
  \end{eqnarray*}
In particular, $\rho$ is non-degenerate and $\psi(\fs) \not\in \GLR{p}{\BQ_p}$
  since $\sqrt{p+1} \not\in \BQ_p$ for $p>3$. By Lemma \ref{nondeglem},   there exists a signed permutation matrix $U$ such that $U \rho(\fs) U\inv =\psi(\fs)$. By Theorem \ref{t:Galois}, $\rho(\fs) \in \GLR{p}{\BQ_p}$, and so is $\psi(\fs)$, a contradiction.
\end{proof}

\section{Applications to Classification}
  \label{Applications: Rank 5 Classification}
\subsection{Rank $5$ Modular Categories}\label{Rank 5}

  In this section we will classify modular categories of rank $5$ as fusion
  subcategories of twisted versions of familiar categories associated to quantum
  groups of type $A$.

  Fix two integers $N\geq 2$ and $\ell>N$.  For any $q$ such that $q^2$ is a
  primitive $\ell$th root of unity we obtain a modular category
  $\CC(\mathfrak{sl}_{N},q,\ell)$ as a subquotient of the category of representations of
  $U_q\mathfrak{sl}_N$.  See \cite{R1} for a survey on the
  construction of such categories, which were first constructed as braided
  fusion categories by Andersen and collaborators and as modular categories
  by Turaev and Wenzl (see the references of \cite{R1}). The fusion rules of
  $\CC(\mathfrak{sl}_{N},q,\ell)$ do not depend on the choice of $q$, \textit{i.e.} for
  fixed $N$ and $\ell$ the categories $\CC(\mathfrak{sl}_{N},q,\ell)$ are all Grothendieck
  equivalent.
  We will denote by $SU(N)_k$ the modular
  category obtained from the choice $q=e^{\pi i/(N+k)}$, i.e.
  $SU(N)_k=\CC(\mathfrak{sl}_{N},e^{\pi i/(N+k)},N+k)$ where $k\geq 1$.  When $\ell$ and $N$
  are relatively prime the category $\CC(\mathfrak{sl}_N,q,\ell)$ factors as a (Deligne)
  product of two modular categories, one of which is (the maximal pointed modular subcategory) of rank $N$ with
  fusion rules like the group $\Z_N$.  For $SU(N)_k$ we will denote the corresponding quotient
  (modular) category by $SU(N)_k/\Z_N$.\footnote{This notation is conventional in
  conformal field theory where the term \emph{orbifold} is used.}

  We will prove:
  \begin{thm}
    \label{thm:fusionrules}
    Suppose $\mcC$ is a modular category of rank $5$.  Then $\mcC$ is Grothendieck
    equivalent to  one of the following:
    \begin{enumerate}
      \item $SU(2)_4$,
      \item $SU(2)_9/\Z_2$,
      \item $SU(5)_1$, or
      \item $SU(3)_4/\Z_3$.
    \end{enumerate}
  \end{thm}
  \begin{proof}
    This follows from \lemmaref{lemma: Rank5Cases} and Propositions \ref{prop: Rank5Z5}, \ref{prop:no(1,-1,-1)},
    \ref{prop: Rank5Z3}, \ref{prop: Rank5Z4}, \ref{prop: Rank5Klein4},
    \ref{prop: Rank5DisjointTranspositions}.
  \end{proof}
  \begin{rmk}
    Although this result only classifies rank $5$ modular categories up to fusion rules, a classification up to equivalence of monoidal categories
    can be obtained using \cite{KazWen}.
     Indeed, by \emph{loc. cit.} Theorem $A_\ell$ modular categories with fusion rules as in (i) resp. (iii) are monoidally equivalent to a Galois conjugate of $SU(2)_4$ followed by a twist of the associativities, resp. a Galois conjugate of $SU(5)_1$ (the non-trivial twists of $SU(5)_1$ have no modular structure).  Modular categories Grothendieck equivalent to $SU(2)_9/\Z_2$ (resp. $SU(3)_4/\Z_3$) are monoidally equivalent to a Galois conjugate of $SU(2)_9/\Z_2$ (resp. $SU(3)_4/\Z_3$)  by Lemma \ref{l:SUequivalence}.

    By \cite[Thm. $A_\ell$]{KazWen} there are \emph{at
    most} $N\varphi(2(k+N))$ (Euler-$\varphi$) inequivalent fusion categories that are Grothendieck
    equivalent to $SU(N)_k$ and \emph{at most} $\varphi(2(k+N))$ for $SU(N)_k/\Z_N$.  The factor of
    $N$ comes from twisting the associativities that is trivial on the quotient $SU(N)_k/\Z_N$ and the $\varphi(2(k+N))$ factor corresponds to a
    choice of a primitive $2(k+N)$th root of unity.  We do not know how many distinct modular categories with these underlying fusion categories there are.
  \end{rmk}

  We first reduce to the case where $\mcC$ is non-integral and self-dual by the following:
  \begin{prop}\rm{\cite[Thms. 3.1 and 3.7]{HR1}}
    Suppose $\mcC$ is a rank $5$ modular category.  Then
    \begin{enumerate}
      \item[(a)] if $\mcC$ is integral then $\mcC$ is Grothendieck equivalent to
      $SU(5)_1$;
      \item[(b)] if $\mcC$ is non-integral and not self-dual then $\mcC$
      is Grothendieck equivalent to $SU(3)_4/\Z_3$.
    \end{enumerate}
  \end{prop}

  We therefore assume $\mcC$ is a non-integral, self-dual modular category of rank
  5 with Frobenius-Schur exponent $N$, and $\rho$ is an \textit{even} modular
  representation of level $n$.  In particular the $S$-matrix has real entries and is
  projectively in $SO(5)$.
  Next we enumerate the possible Galois groups $\Gal(\mcC)$ for rank $5$
  modular categories $\mcC$.

  \begin{lem}
    \label{lemma: Rank5Cases}
    Suppose $\mcC$ is a self-dual non-integral modular category of rank $5$.  Then
    up to reordering the isomorphism classes of simple objects we have
    $\Gal(\mcC)$ is cyclic and generated by one of the following:
    $\(0\;1\)$, $\(0\;1\;2\)$, $\(0\;1\;2\;3\)$, $\(0\;1\;2\;3\;4\;5\)$,
    $\(0\;1\)\(2\;3\)$; or it is a Klein 4 group given by either
    $\langle\(0\;1\),\(2\;3\)\rangle$, or $\langle (0\;1)(2\;3),(0\;2)(1\;3)\rangle$
  \end{lem}
  \begin{proof}
    Since we have assumed $\mcC$ is not integral, Lemma \ref{integrallemma} implies
    $0$ is not fixed by $\Gal(\mcC)$.  Relabeling the simple objects if necessary
    we arrive at a list of possible groups. The groups
    $\langle\(0\;1\;2\)\(3\;4\)\rangle$ and $\langle\(0\;1\)\(2\;3\;4\)\rangle$ can
    be excluded by Lemmas \ref{l:c61} and \ref{l:c62}.
  \end{proof}

  First observe that the case $\Gal\(\mcC\)\cong\BZ_5\cong\langle (0\;1\;2\;3\;4)\rangle$ has been
  considered in \thmref{t:irr}.
  \begin{prop}
    \label{prop: Rank5Z5}
    If $\mcC$ is a rank $5$ modular category with $(0\;1\;2\;3\;4)\in\Gal(\mcC)$ then $\mcC$ is equivalent to $SU(2)_{9}/\BZ_2$ as fusion categories.
  \end{prop}

  Next we will consider the case that $\Gal(\mcC)=\langle (0\;1)\rangle$.
  The following lemma will be useful.
  \begin{lem}
    \label{l:vanishing sum}
    Let $a, b$ be non-zero rational integers. Suppose
    \begin{equation}
      \label{eq:vanishing sum}
      0=a+bi+ c_\a \a+ c_\b \b
    \end{equation}

    for some non-zero rational integers $c_\a$, $c_\b$ and roots of unity  $\a$, $\b$ with
    $\ord(\a) \le \ord(\b)$. Then $\a =\pm 1$, $\b=\pm i$ and
    \begin{equation*}
      a+ \a c_\a = 0\,, \quad b-i \b_2 c_\b = 0
    \end{equation*}
  \end{lem}
  \begin{proof}
    If $\a, \b \in \BQ(i)$, then $\a,\b$ are fourth roots unity. The $\BQ$-linear
    independence of $\{1, i\}$ implies that $\a = \pm 1$ and $\b = \pm i$. Thus,
    the remainder equalities follow immediately. Therefore, it suffices to show
    that $\a,\b \in \BQ(i)$.

    Suppose that  $\a$ or $\b$ is not in $\BQ(i)$. Then \eqref{eq:vanishing
    sum} implies that
    $[\BQ(i, \a):\BQ(i)] = [\BQ(i,\b):\BQ(i)]$. Hence, both $\a,\b$ are not in
    $\BQ(i)$. Note that  $\a, \b$ are $\BQ(i)$-linearly independent otherwise $\a,
    \b \in \BQ(i)$. By \cite[Thm. 1]{CJ}, there exist $x, y \in \{\a, \b\}$  such
    that $x, y/i$ have squarefree orders,  and
    \begin{equation*}
      a+ c_x x =0, \quad i b+  c_y y =0\,.
    \end{equation*}

    These equations force $\a= x =\pm 1$ and $\b=y = \pm i$, and hence $\a, \b \in
    \BQ(i)$, a contradiction.
  \end{proof}

  We have:  \begin{prop}
    \label{prop:no(1,-1,-1)}
    If $\Gal(\mcC)=\langle (0\;1)\rangle$ then $\mcC$ is Grothendieck equivalent to
    $SU(2)_4$.
  \end{prop}
  \begin{proof}
  Suppose $\mcC$ is a rank 5 modular category with $\Gal(\mcC)=\langle (0\;1)\rangle$.
  By \eqref{eq:2} and \lemmaref{201304261618}, the $S$-matrix is of the form
  \begin{equation*}
    S =\mtx{1 & d_1 & d_2 & d_3 & d_4 \\
    d_1& 1 & \e_2 d_2 & \e_3 d_3 &\e_4 d_4 \\
    d_2 & \e_2 d_2 & S_{22} & S_{23} & S_{24}\\
    d_3 & \e_3 d_3 & S_{32} & S_{33} & S_{34}\\
    d_4 & \e_4 d_4 & S_{42} & S_{43} & S_{44}
    }
  \end{equation*}

  where $\e_i = \pm 1$ and
  $\(\e_{2},\e_{3},\e_{4}\)\neq\(-1,-1,-1\)$ or $\(1,1,1\)$.  After renumbering, we may therefore assume that $(\e_2,\e_3,\e_4)\in\{(1,1,-1),(1,-1,-1)\}$.

    Suppose that $(\e_2, \e_3, \e_4) =(1,1,-1)$. We first use Lemma \ref{201304261618} to conclude that  $S_{24}=S_{34}=0$ and
    orthogonality of the first and last columns of $S$ to obtain $S_{44}=d_1-1$. Then we use
    the twist equation (\ref{twisteq})
    for $(j,k) = (2,4)$, $(0,4)$ and $(4,4)$ to obtain
    \begin{eqnarray}
      \label{eq:tw1} 0=p^{+} S_{24} &=&\theta_2 \theta_4 (d_2 d_4-\theta_1 d_2d_4),\\
      \label{eq:tw2}  p^{+} d_4 &=& \theta_4 (d_4-\theta_1 d_1d_4+\theta_4 d_4 (d_1-1
      )),\\
      \label{eq:tw3}  p^{+} (d_1-1) &=& \theta^2_4 (d^2_4+\theta_1 d^2_4+\theta_4 (d_1-1
      )^2)\,.
    \end{eqnarray}

    It follows immediately from \eqref{eq:tw1} that $\theta_1=1$ and hence, by
    \eqref{eq:tw2},
    \begin{equation*}
      p^{+}=(d_1-1)\theta_4(\theta_4-1)\,.
    \end{equation*}

    Therefore,
    \begin{equation*}
      D^2=2(d_1-1)^2(1-Re(\theta_4))  \text{ and } d_1 \not\in \BQ\,.
    \end{equation*}

    Since $\frac{S_{i4}}{d_4}$ is an algebraic integer fixed by $\Gal(\mcC)$,
    $\frac{S_{i4}}{d_4}  \in \BZ$ for all $i$. In particular,
    \begin{equation*}
      n_{44}=\frac{d_1-1}{d_4} \in \BZ.
    \end{equation*}

    and
    \begin{equation*}
      D^2=(2+n^2_{44}) d_4^2\,.
    \end{equation*}
    It follows from \eqref{eq:tw3} that

    \begin{equation*}
      p^{+}  = (d_1-1 )( \frac{2}{n_{44}^2}\theta^2_4 +\theta_4^3 )
    \end{equation*}

    and this implies $\frac{2}{n_{44}^2}\theta^2_4
    +\theta_4^3=\theta_4(\theta_4-1)$ or
    \begin{equation*}
      \theta_4^2 +(\frac{2}{n_{44}^2}-1)\theta_4 +1=0\,.
    \end{equation*}

    Thus, $[\BQ(\theta_4):\BQ]\le 2$ and so $\theta_4 \in \ol\mu_4 \cup \ol\mu_6$. Note that $\frac{2}{n_{44}^2}-1 \not\in\{0,
    -1, \pm 2\}$. Therefore, $\theta_4 \not \in \ol\mu_4 \cup \mu_6$. Thus, $\theta_4 \in \mu_3$  and $n_{44}=\pm 1$.
    Now, we find $D=\sqrt{3}|d_1-1|$, $p^{+}=-2 i Im(\theta_4) (d_1-1)=\pm i \sqrt{3}
    (d_1-1)$.
    Since the $S$-matrix is real, $\CC$ is self-dual.  So by \cite[Thm. 2.7(5)]{RSW}, $D\in \BK_\mcC$ and so $\sqrt{3} \in \BK_\mcC$. Since $[\BK_\mcC:\BQ]=2$, $\BK_\mcC
    =\BQ(\sqrt{3})$.

    We now return to the equation
    \begin{equation}\label{eq4.5}
      p^{+}= 1+d_1^2 + (d_1-1)^2\theta_4 +\theta_2 d_2^2+ \theta_3 d_3^2.
    \end{equation}
Multiplying this equation by $2/d_1$, using $p^{+}=-(\theta_4-\overline{\theta_4})(d_1-1)$ and $\theta_4+\overline{\theta_4}+1=0$ we can reexpress equation (\ref{eq4.5}) by
    \begin{equation}
      \label{eq:theta-rel}
      0=2 +  \Tr(d_1)  + 2i Im(\theta_4)(d_1-1/d_1) + 2\theta_2 N(d_2) + 2\theta_3 N(d_3) \,,
    \end{equation}
where $N(d_i)=d_i^2/d_1$ for $i=2,3$, and $\Tr(d_1)=d_1+1/d_1.$
   Note that $2+ \Tr(d_1)$,  $N(d_2)$ and $N(d_3)$ are
    non-zero integers. Since $\BZ[\sqrt{3}]$ is the ring of algebraic integers in $\BQ(\sqrt{3})$, $2 Im(\theta_4) (d_1-1/d_1)$ is also a  non-zero integer.
    We may simply assume $\ord(\theta_2) \le \ord(\theta_3)$. By
    \lemmaref{l:vanishing sum}, $\theta_2 = \pm 1$ and
    \begin{equation*}
      2+ \Tr(d_1) + 2\theta_2 N(d_2)=0\,.
    \end{equation*}

    Since $2+ \Tr(d_1)$ and  $2 N(d_2)$ are positive, $\theta_2=-1$ and $2 d_2^2
    =(d_1+1)^2$. Thus, $\sqrt{2}  =\pm \frac{d_1+1}{d_2} \in \BQ(\sqrt{3})$, a
    contradiction.

 Therefore we must have $(\e_2, \e_3, \e_4) = (1,-1,-1)$ and, by Lemma \ref{201304261618},
    \begin{equation*}
      S =\mtx{1 & d_1 & d_2 & d_3 & d_4 \\
      d_1& 1 &  d_2 & - d_3 &- d_4 \\
      d_2 &  d_2 & S_{22} & 0 & 0\\
      d_3 & - d_3 & 0 & S_{33} & S_{34}\\
      d_4 & - d_4 & 0 & S_{34} & S_{44}
      }\,.
    \end{equation*}

    By the orthogonality of the columns of $S$, $S_{22}=-(d_1+1)$. Since
    $\frac{S_{22}}{d_2}$ is fixed by $\Gal(\CC)$,
    \begin{equation*}
      n_{22}=\frac{S_{22}}{d_2} =  \frac{-(d_1+1)}{d_2} \in \BZ.
    \end{equation*}

    By Lemma \ref{integrallemma} the vector of FP-dimensions is in one of the first two rows, so $d_1, d_2>0$ and
    $n_{22}<0$. We now apply the twist equation (\ref{twisteq})
    for $(j,k) = (2, 0)$, $(2,2)$ and $(2,3)$ to obtain
    \begin{eqnarray}
      \label{eq2:tw1} p^{+} d_2 &=& \theta_2 (d_2+\theta_1 d_1d_2-\theta_2 d_2
      (d_1+1)),\\
      \label{eq2:tw2}  -p^{+} (d_1+1) &=& \theta^2_2 (d^2_2+\theta_1 d^2_2+\theta_2
      (d_1+1 )^2),\\
      \label{eq2:tw3}   0=p^{+} S_{23} &=&\theta_2 \theta_3 (d_2 d_3-\theta_1 d_2d_3)\,.
    \end{eqnarray}
    The equation \eqref{eq2:tw3} implies $\theta_1=1$, and so equations
    \eqref{eq2:tw1}, \eqref{eq2:tw2} become
    \begin{eqnarray}
      \label{eq2:tw4} \frac{p^{+}}{d_1+1}  &=& \theta_2 (1-\theta_2),\\
      \label{eq2:tw5}  -\frac{p^{+}}{d_1+1} &=& \theta^2_2 (\frac{2} {n^2_{22}}+\theta_2
      ).
    \end{eqnarray}

    Thus, $\theta_2$ satisfies the quadratic equation
    \begin{equation*}
      \theta^2_2 + (\frac{2} {n^2_{22}}-1)\theta_2 +1=0\,.
    \end{equation*}
    Since $n_{22}$ is a negative integer, $\frac{2} {n^2_{22}}-1 \ne 0, -1, \pm 2$.
    Therefore, $\theta_2\in\mu_3$, $n_{22}=-1$ and
    $d_2=d_1+1$. Moreover,
    \begin{equation*}
      p^{+}= 2 i Im (\theta_2) (d_1+1) = \pm i \sqrt{3} (d_1+1), \quad \text{and}\quad
      D^2=3 (d_1+1)^2.
    \end{equation*}

    In particular, $D=\sqrt{3}(d_1+1)$. Since $\mcC$ is self-dual, by \cite[Thm. 2.7(5)]{RSW}, $\sqrt{3} \in \BK_\mcC$
    and hence $\BK_\mcC=\BQ(\sqrt{3})$.

    We now return to the equation
    \begin{equation}
      p^{+}= 1+d_1^2 + (d_1+1)^2\theta_2 +\theta_3 d_3^2+ \theta_4 d_4^2
    \end{equation}

    which can be rewritten (in a similar way as in eqn. (\ref{eq4.5})) as
    \begin{equation}
      \label{eq:theta-rel2}
      0=-2 +  \Tr(d_1)  + 2i Im(\theta_2)(d_1-1/d_1) + 2\theta_3 \frac{d_3^2}{d_1} +
      2\theta_4 \frac{d_4^2}{d_1} \,.
    \end{equation}
    Without loss of generality, we may simply assume
    $\ord(\theta_3)\le \ord(\theta_4)$.
   We first prove that $d_1\in \BQ$. Suppose not. Since  $-2 +  \Tr(d_1)$,  $2i Im(\theta_2)(d_1-1/d_1)$,  $\frac{d_3^2}{d_1}$  and $\frac{d_4^2}{d_1}$
    are non-zero integers,  by Lemma \ref{l:vanishing sum}, we
    find $\theta_3=\pm 1$ and
    \begin{equation*}
      -2 +  \Tr(d_1) + 2\theta_3 \frac{d_3^2}{d_1} =0\,.
    \end{equation*}

    Since $-2 +  \Tr(d_1), \frac{d_3^2}{d_1}> 0$, $\theta_3=-1$ and $2 d_3^2
    =(d_1-1)^2$. However, this implies
    $\sqrt{2} =\pm\frac{d_1-1}{d_3} \in \BQ(\sqrt{3})$, a contradiction.  Therefore $d_1\in\BQ$.

    Since $1/d_1$ is a Galois conjugate of $d_1$, $d_1=1$. Now, \eqref{eq:theta-rel2} becomes $0=\theta_3 \frac{d_3^2}{d_1} + \theta_4
    \frac{d_4^2}{d_1}$
    or
    \begin{equation*}
      \theta_3/\theta_4 = -\frac{d_4^2/d_1}{d_3^2/d_1} \in \BQ\,.
    \end{equation*}

    This forces $\theta_3=-\theta_4$ and $d_4^2=d_3^2$. Since
    \begin{equation*}
      12 = D^2 = 1+ 1+ 2^2+ 2d_3^2,
    \end{equation*}

    we obtain $d_3=\pm \sqrt{3}$.

    Suppose $d_3 = \nu_1 \sqrt{3}$ and $d_4 = \nu'_1 \sqrt{3}$ for some signs
    $\nu_1, \nu'_1$. The fusion rule $N_{23}^4 = \nu_1 \nu'_1$ implies
    $\nu_1=\nu'_1$. It follows from the orthogonality of the $S$-matrix that
    \begin{equation*}
      0=  S_{33} +  S_{34} =   S_{34} +  S_{44}= 6 + (S_{33} +  S_{44})S_{34}, \quad
      S_{33}^2+S_{34}^2= 6 =S_{34}^2+S_{44}^2\,.
    \end{equation*}

    Therefore,
    \begin{equation*}
      S_{33}=S_{44}=-\nu_2 \sqrt{3},  \quad S_{34}=\nu_2 \sqrt{3}
    \end{equation*}

    for any sign $\nu_2$. We find
    \begin{equation*}
      S =\mtx{1 & 1 & 2 & \nu_1 \sqrt{3} & \nu_1 \sqrt{3} \\
      1& 1 & 2 & -\nu_1 \sqrt{3} & -\nu_1 \sqrt{3} \\
      2 &  2 & -2 & 0 & 0\\
      \nu_1 \sqrt{3} &-\nu_1 \sqrt{3} & 0 & -\nu_2 \sqrt{3} & \nu_2
      \sqrt{3}\\
      \nu_1 \sqrt{3} & -\nu_1 \sqrt{3}& 0 & \nu_2 \sqrt{3} & -\nu_2 \sqrt{3}
      }\,.
    \end{equation*}

    On can check directly the four possible $S$-matrices of $\mcC$ generate the same
    fusion rules using the Verlinde formula.  These fusion rules coincide with
    those of $SU(2)_4$.

    We return to the twist equation (\ref{twisteq}) with $(j,k)=(0,3)$ to obtain
    \begin{equation*}
      \theta_3^2= -\nu_2 2 i \,Im(\theta_2)/\sqrt{3} = -\nu_2 \nu_3 i
    \end{equation*}

    where $\nu_3 = \pm 1$ is determined by $\theta_2 = e^{\nu_3 2 \pi i/3}$.
    One can check directly that for any $\theta_2 \in \mu_3$ and
     $\theta_3 \in \mu_8$ satisfying the above equation, the twist
    equation will hold for  $T=\diag(1,1,\theta_2, \theta_3, -\theta_3)$. Thus,
    there are 16 possible pairs of $S$ and $T$-matrices for $\mcC$.
  \end{proof}

  \begin{remark}
    Each of the 16 possible pairs of $S$ and $T$ matrices are realized.  By
    applying a Galois automorphism we may assume $\nu_1=1$, that is,
    $d_i=\FPdim(V_i)$ for all $i$.  Then the $8$ pairs $(S,T)$ with $\nu_1=1$ lead to $4$ distinct modular categories after relabeling the last two objects, which all appear
    in \cite[Example 5D]{GNN1}.
  \end{remark}

  Next we show
  \begin{prop}
    \label{prop: Rank5Z3}
    If $\mcC$ is a self-dual modular category of rank $5$, then $\Gal(\mcC) \not\cong
    \BZ_3$.
  \end{prop}
  \begin{proof}

  Suppose that $\Gal(\CC)\cong \BZ_3$ and $\rho$ is an \emph{even} level $n$ modular representation of $\CC$ (cf. Remark \ref{r1}).  Since $(\BK_\CC, \BQ_n)$ is admissible,  by
     \propref{201301221709} we have either $7\mid n \mid 24\cdot 7$ or $9\mid n \mid 8 \cdot 9$. We will eliminate these two possibilities.

    Suppose $7 \mid n \mid 24 \cdot 7$. Then $\rho$ has an irreducible subrepresentation $\rho_1$
    of level $7 f$ where $(7,f)=1$ and $7 f \mid n$.  Thus, $\rho_1\cong \xi
    \o \phi$ for some irreducible representations $\xi$ and $\phi$ of levels $7$ and
    $f$ respectively. By \cite[Table 1]{Eh1}, $\deg \xi=3, 4$ and
    hence $\deg \phi=1$.

    If $\deg \xi=3$, then, by Table \ref{tableA1}, its $\ft$-spectrum is a subset  of $\mu_7$. This is not possible by Corollary \ref{c:1}. If $\deg
    \xi=4$, then by Table \ref{tableA1}  $\xi$ is odd; this contradicts Lemma
    \ref{l:4.25}.

    Now suppose $9 \mid n \mid 8\cdot 9$. Then $\rho$ has an irreducible subrepresentation $\rho_1$
    of level $9 f$ where $(3,f)=1$ and $9 f \mid n$.  Thus,
    $\rho_1\cong \xi \o \phi$ for some irreducible representations $\xi$ and $\phi$ of
    levels $9$ and $f$ respectively. By \cite[Table 2]{Eh1},
    $\deg \xi=4$ and hence $\deg \phi=1$. Thus, $\rho\cong \phi' \oplus (\phi\o \xi) $ for some degree 1 representation $\phi'$.

    By Lemma \ref{l:4.25}, $\xi$, $\phi', \phi$ are all even. Therefore, $(\phi')^* \o
    \rho \cong \chi_0 \oplus ((\phi')^* \o\phi \o \xi)$ where $\chi_0$ is the
    trivial representation of ${\rm SL}(2,\Z)$. Note that $(\phi')^* \o \rho$ is isomorphic
    to another even modular representation $\rho'$ of $\mcC$.

    Let $\rho_1=((\phi')^* \o\phi \o \xi)$. By Lemma \ref{l:1}, $\rho_1(\ft)$ has an
    eigenvalue $1$ and so $((\phi')^* \o\phi )^3 = \chi_0$. Therefore, $\rho_1$ is a
    level 9 irreducible representation of ${\rm SL}(2,\Z)$.
    By \cite[Table A3]{Eh1},  $\rho_1$ is isomorphic to $R$ or $R^*$ defined by
    \begin{equation*}
      R(\fs):=\frac{2}{3} \mtx{
      s_1 & s_5 & s_7 & s_6\\
      s_5 & -s_7 & -s_1 & s_6\\
      s_7 & -s_1 & s_5 & -s_6\\
      s_6 & s_6 & -s_6 & 0}
      ,\quad
      R(\ft):=\diag( \zeta, \zeta^7, \zeta^4, 1)
    \end{equation*}

    with $s_j = \sin(\pi j/18)$ and $\zeta=\exp(2 \pi i/9)$. Note that
    $R(\fs)=R^*(\fs)$.

    Since $\rho' \cong \rho_1 \oplus \chi_0$,  $\rho$ is of level 9 and $\rho'(\fs)$
    is a matrix over $\BQ_9$. Let $R'=R\oplus \chi_0$. Then, there exists a permutation matrix $P$ and a unitary matrix $U$  such that
    \begin{equation*}
      P \rho'(\ft) P\inv= R^\prime (\ft) =\diag(\w_1, \w_2, \w_3, 1, 1)=U\inv R^\prime(\ft)U,
    \end{equation*}
    where $\w_1, \w_2, \w_3$ are distinct 9-th roots of unity.
    Moreover,
    \begin{equation*}
      P \rho'(\ft) P\inv = U\inv R^\prime(\ft)U, \quad
      P \rho'(\fs) P\inv = U\inv R'(\fs) U.
    \end{equation*}

     This
    implies that $U$ is of the form
    \begin{equation*}
      \mtx{
      u_1 & 0 & 0 & 0 & 0\\
      0 & u_2 & 0 & 0 & 0\\
      0 & 0 & u_3 & 0 & 0\\
      0 & 0 & 0 & a & b\\
      0 & 0 & 0 & -\ol b & \ol a}
    \end{equation*}

    where $|u_1|=|u_2|=|u_3|=1$ and $|a|^2+|b|^2=1$. We can further assume that
    $u_1=1$. Since
    $P \rho'(\fs) P\inv$ is symmetric, $u_1, u_2$ are $\pm 1$ and  $a, b $ are real.
    Thus, the $(1,4)$ and $(1,5)$ entries of $P \rho'(\fs) P\inv$ are
    $\frac{a}{\sqrt{3}}$ and $\frac{b}{\sqrt{3}} \in \BQ_9$. This implies $ab \in
    \BQ_9$ and $ \frac{1}{3}+ \frac{2ab}{\sqrt{3}}=(\frac{a}{\sqrt{3}}+\frac{b}{\sqrt{3}})^2 \in
    \BQ_9\,.$ Hence, $\sqrt{3}\in \BQ_9$ but this contradicts that the conductor of $\sqrt{3}$
    is 12.
  \end{proof}

  \begin{prop}
    \label{prop: Rank5Z4}
   If $\CC$ is a self-dual non-integral modular category of rank $5$ then $\Gal(\mcC) \not\cong \BZ_4$.
  \end{prop}
  \begin{proof}

    Assume to the contrary.  Let $\rho$ be a level $n$ even modular representation of
    $\mcC$ (see Remark \ref{r1}), and $\s\in \GalQ{n}$ be such that $\hs=(0\;1\;2\;3)\in\mathfrak{S}_5$ is a generator of the image of $\Gal(\mcC)$ in $\mathfrak{S}_5$.
    It follows from \propref{201301221709} that the level $n$ of $\rho$ satisfies
    one of the following cases:
    \begin{enumerate}
    \item  $5 \mid n \mid 24\cdot 5$,
    \item $16 \mid n \mid 3\cdot 16$,
    \item $32 \mid n \mid 3\cdot 32$.
    \end{enumerate}

    By \cite[Table 7]{Eh1}, the smallest irreducible representation of level 32 is
    6-dimensional. Therefore, case (iii) is impossible.

    In cases (i) and (ii) we find that $\Z_n^*=\Gal(\BQ_n/\BQ)$ has exponent $4$, so that
    $\sigma^4=\id$.  Applying Theorem \ref{t:Galois}(iii) we find that $\rho(\mathfrak{t})=t=\diag(z,\sigma^2(z),z,\sigma^2(z),w)$ where
    $w \in \ol\mu_{24}$ and $z$ is a root of unity such that $5 \mid \ord (z) \mid 24\cdot 5$ or $16 \mid \ord(z) \mid 3\cdot 16$.  By \cite{TubaWenzl}  $\rho$ cannot have an irreducible subrepresentation
    of dimension more than $3$. Moreover, $\rho$ cannot have $1$-dimensional subrepresentations: $z$ cannot be the image of $\mathfrak{t}$ in a $1$-dimensional ${\rm SL}(2,\Z)$-representation as $z^{24}\neq 1$, and by Lemma \ref{l:1} $w$ cannot be the image of $\mathfrak{t}$ in a $1$-dimensional representation either, since $w$ is distinct from $z$ and $\sigma^2(z)$.

We can therefore conclude that $\rho$ is a direct sum of even irreducible representations $\rho_2$ and $\rho_3$ of degrees 2 and 3 respectively. The corresponding partition of the $\ft$-spectrum of $\rho$  is  $\{\{z, \s^2(z)\}, \{z, \s^2(z), w\}\}$. In particular, the levels of these representations are multiple of $\ord(z)$.   If $16 \mid n \mid 48$, then $16 \mid \ord(z)$ and there must be an irreducible representation of level 16 and degree 2. By  Table \ref{tableA1}, this is not possible and we conclude that  $5\mid n\mid 24\cdot 5$.

The representation $\rho_3 \cong \psi \o \chi$ for some irreducible representations $\psi$ of degree 3 and level 5, and $\chi$ of degree 1. By Table \ref{tableA1}, $\psi$ is even, and so must be $\chi$. Thus, the spectrum of $\rho_3(\ft)$ is $\{w \zeta, w/\zeta, w\}$ for some $\zeta \in \mu_5$ and $w \in \ol\mu_6$.  This forces the $\ft$-spectrum of $\rho_2$ to $\{w \zeta, w/\zeta\}$ and so $\rho_2 \cong \psi' \o \chi$ for some irreducible representations $\psi'$ of degree 2 and level 5. By Table \ref{tableA1}, $\psi'$ is odd, and so must be $\rho_2$. This contradicts that $\rho$ is even.
\end{proof}

  \begin{prop}
    \label{prop: Rank5Klein4}  If $\CC$ is a self-dual non-integral modular category of rank $5$ then
    $\Gal(\mcC)\not\cong\langle (0\;1),(2\;3)\rangle$
  \end{prop}
  \begin{proof}
    Suppose $\Gal(\mcC)=\langle \s,\tau\rangle$  such that $\hs=(0\;1)$ and
    $\hat{\tau}=(2\;3)$. For notational convenience we set $\delta_i=\e_\tau(i)$ and
    $\e_i=\e_\sigma(i)$. Galois symmetry (with respect to $\sigma$) applied to
    $S_{i,(i+1)}$ gives us the following condition for each $i\geq 2$: \emph{either
    $S_{i,(i+1)}=0$ or $\e_i=\e_{i+1}$}.  Similarly, Galois symmetry with respect to
    $\tau$ applied to $S_{0i}=d_i$ gives us: $\delta_0=\delta_1=d_4$ and
    $\delta_2=\delta_3$.  With this in mind we set $e_1=\e_0\e_2$,
    $e_2=\e_0\e_3$, $e_3=\e_0\e_4$ and $a=\delta_0\delta_2$.  Applying $\sigma$ and $\tau$ we obtain:
    $$S=\mtx{ 1&d_1&d_2&ad_2&d_4\\d_1&1&e_1d_2&e_2
    ad_2&e_3d_4\\d_2&e_1d_2&S_{22}&S_{23}&S_{24}\\ a
    d_2&e_2ad_2&S_{23}&S_{22}&a
    S_{24}\\d_4&e_3d_4&S_{24}&aS_{24}&S_{44}}.$$
  Since $\sigma(d_2)=e_1d_2/d_1=e_2d_2/d_1$ we immediately see that $e_2=e_1$.
    Orthogonality then implies that either $S_{24}=0$ or $e_1=e_3$, and
    $$
    \{d_1+1/d_1, d_2^2/d_1, d_4^2/d_1, (S_{22}+a S_{23})/d_2, S_{22}S_{23}/d_2^2\} \subset \Z.
    $$

    We claim that $S_{24}=0$.  If $e_i=1$ for all $i$ then orthogonality of the
    first two rows gives: $2=-2d_2^2/d_1-d_4^2/d_1$ with $d_1$ negative.
    If $e_i=-1$ for each $i$ then orthogonality of
    the first two rows gives us: $2=2d_2^2/d_1+d_4^2/d_1$.
    In either case, we have: $2=2x+y$ for some $x,y\geq 1$, which is absurd.

    So we may assume that $S_{24}=0$ and $-e_3=e_1=e_2$.  In particular, the
    $FP$-dimension must be one of the first two rows.  Therefore, $a=1$ and $d_1>0$.  Orthogonality now implies:
    \begin{eqnarray}
      1+e_3&=&0,\\
      1+e_1d_1+S_{22}+S_{23}&=&0,\\
      1+e_3d_1+S_{44}&=&0\,.
    \end{eqnarray}

    Thus $e_2=e_1=1=-e_3$, $S_{44}=d_1-1$ and $1+d_1+S_{22}+S_{23}=0$. Note
    that this implies $M=(1+d_1)/d_2\in\BZ$.

    Using the twist equation (\ref{twisteq}) we proceed as in the proof of
    \propref{prop:no(1,-1,-1)} to obtain: $d_4=\pm (d_1-1)$, $p^{+}=\pm i\sqrt{3}(d_1-1)$ and
    $p^{+}/p^{-}=-1$ and $\theta_4 \in \mu_3$. Thus,
    we have
    $$
    p^++p^-=2(1+d_1^2)+2 d_2^2 Re(\theta_2+\theta_3) -(d_1-1)^2 = 0
    $$
    or  $2 Re(\theta_2+\theta_3) = -M^2$.
    Setting $N=d_2^2/d_1$, we obtain the Diophantine equation $(M^2-2)N=6$ from
    orthogonality of the first two rows (\textit{i.e.} $d_1^2-4d_1+1=2d_2^2$).
    Since each of $M$ and $N$ are positive integers we obtain $(M,N)=(2,3)$ as the only solution.  Therefore, $Re(\theta_2+\theta_3)=-2$. Hence $\theta_2=\theta_3=-1$, and $\BF_T = \BQ_3$. However, we also find $d_1 = 5+2 \sqrt{6} \not \in \BQ_3$ which contradicts Theorem \ref{t:cong1}.
  \end{proof}

  Two cases remain: either $\Gal(\CC)$ is generated by $\hs=\(0\;1\)\(2\;3\)$, or contains $\hs$ and is isomorphic to $\Z_2\times\Z_2$ acting transitively on
  $\{0,1,2,3\}$ fixing the label $4$. In either case, $\exp(\Gal(\CC))=2$ so $\BK_\CC$ is a multi-quadratic extension of $\BQ$.
  \begin{prop}
    \label{prop: Rank5DisjointTranspositions}
    There is no self-dual non-integral modular category $\mcC$ of rank 5 such that
    every non-trivial element of $\Gal(\mcC)$ is a product of two disjoint
    transpositions.
  \end{prop}

  \begin{proof}
    In the following series of reductions, we will show that $\FSExp\(\mcC\)$
    can only be 2,3,4,6. In particular, $\mcC$ is integral by \cite[Thm.
    3.1]{BR1}, a contradiction.

Let $\rho$ be an even modular representation of $\CC$ of level $n$.
Without loss of generality we may assume that $\hs=(0\;1)(2\;3)\in\Gal(\CC)$ for some $\s \in \Gal(\BQ_n/\BQ)$. By Theorem \ref{t:Galois} (Galois symmetry),
\begin{equation} \label{eq:t_form}
\rho(\ft)= t=\diag(t_0, \s^2(t_0), t_2, \s^2(t_2), t_4)\,.
\end{equation}
Moreover, $s= \rho(\fs)$ is a real symmetric matrix in $GL(5, \BQ_n)$ of order 2 . Since $\tau^2(t_4) = t_{\htau(4)} =t_4$ and $\tau(s_{i4}/s_{04}) = s_{i \htau(4)}/s_{0\htau(4)} =s_{i4}/s_{04}$
for all $\tau\in \AQab$, $t_4 \in \ol\mu_{24}$ and $s_{i4}/s_{04} \in \BZ$ for all $i=0,\dots,4$.

 By \corref{201302011406},
    $n \mid 240$. We first show that $n \mid 48$, \emph{i.e.} $5\nmid n$.

    Suppose $5\mid n$. Then $\rho \cong (\xi_5 \o \chi) \oplus \rho_1$ where
    $\rho_1$ is an even subrepresentation of $\rho$, $\xi_5$ and $\chi$ are
    irreducible representations of $\SL{5}$ and $\SL{48}$ respectively, and $\xi_5$ is of level 5. Since $\deg\xi_5 \ge 2$, $\deg \chi \le 2$. However, if $\deg\chi=2$, then $\deg \xi_5 =2$ and $\deg \rho_1=1$. By Table \ref{tableA1}, the $\ft$-spectrum of $\xi_5$ is $\{\zeta, \ol \zeta\}$ for some $\zeta_5 \in \mu_5$. Thus the orders of the eigenvalues of $\xi_5 \o \chi$ are multiple of  $5$, and so the $\ft$-spectra of $\xi_5 \o \chi$ and $\rho_1$ are disjoint; this contradicts Lemma \ref{l:1}. Therefore, $\chi$ is linear.

    Now, we set $\rho'$ be the modular representation of $\CC$ equivalent to $\chi\inv \o \rho$ and $\rho_1' = \chi\inv \o \rho_1$.  Then $\rho' \cong \xi_5  \oplus \rho_1'$.

    If $\deg \xi_5 = 5$, then
    $\xi_5\cong \rho'$, but this contradicts Lemma \ref{l:inadp}. Therefore,
    $\deg \xi_5 < 5$.

    If $\deg \xi_5 = 4$, then the $\ft$-spectrum of $\xi_5$ is equal to $\mu_5$ but $\rho'_1$ is linear.  Thus, $\xi_5$ and
    $\rho'_1$ have disjoint $\ft$-spectra.  Therefore, $\deg \xi_5 = 4$ is not possible.

    If $\deg \xi_5 = 2$, then $\xi_5$ is odd and the $\ft$-spectrum of $\xi_5$ is $\{\w, \ol \w\}$ for some $\w \in \mu_5$. Thus, $\chi$ is odd and so are $\rho'$ and $\rho'_1$. If $\rho'_1$ is reducible, then $\rho'_1 \cong
    \rho'_2 \oplus \rho'_3$ for some representations $\rho'_2$ and $\rho'_3$. We may assume $\deg \rho'_3=1$. By Lemma \ref{l:1}, the $\ft$-spectrum of $\rho'_2$ must contain $\w$ or $\ol\w$. Therefore, $\rho'$ is also irreducible and has the same $\ft$-spectrum $\xi_5$. However, this means
   $\rho'_3$ and $\xi_5 \oplus \rho'_2$
    have disjoint $\ft$-spectra.  Therefore, $\rho_1$ must be irreducible and the
    $\ft$-spectra of $\rho'_1$ and $\xi_5$ are not disjoint. This implies $\rho'_1$ is of level 5, and it must be even, a contradiction. Therefore, $\deg \xi_5 \ne 2$.

    If $\deg \xi_5 = 3$, then $\xi_5$ is even and so are $\chi$ and $\rho'$. We may assume $\rho \cong \xi_5 \oplus \rho_1$ by replacing $\rho$ with $\rho'$ if necessary. The $\ft$-spectrum of $\xi_5$
    is $\{\w, \ol \w, 1\}$ for some $\w \in \mu_5$.

    If  $\rho_1$ is reducible, then $\rho_1$ is a direct sum $\chi_1 \oplus \chi_2$ of linear characters. In view of Lemma \ref{l:1}, both of $\chi_1$ and $ \chi_2$ are the trivial character and so $\rho$ is of level 5.  If $\rho_1$ is irreducible, then the $\ft$-spectrum of $\rho_1$ cannot contain $\w$ or $\ol \w$ for otherwise $\rho_1$ is the level 5 degree 2 irreducible representation which is odd. Thus, 1 is an eigenvalue of $\rho_1(\ft)$ and so $\rho_1$ is the level 2 degree 2 irreducible representation  with $\ft$-spectrum $\{1, -1\}$. In particular, $\rho$ is of level 10. Hence, by \eqref{eq:t_form}, we find
$$
\rho(\ft) = t =\diag(\w, \ol\w, 1,1,\pm 1) \text{ or }\diag(1, 1, \w,\ol \w, \pm 1)
$$
for both cases of $\rho_1$. Moreover, $\BF_t =\BQ_5$ and $\BF_s$ is a real subfield of $\BF_t$. Therefore, $\BF_s = \BQ(\sqrt{5})$. Since both generators of $\Gal(\BQ_5/\BQ)$ have the same non-trivial restriction on $\BQ(\sqrt{5})$, we can assume $\s: \w \mapsto \w^2$ and $\hs
    = (0\;1)(2\;3)$.  By the twist equation (\ref{twisteq}), we find
    \begin{equation} \label{eq:44}
      s_{44} = s_{04}^2 t_0 + s_{14}^2 t_1 + s_{24}^2 t_2 +  s_{34}^2 t_3+ s_{44}^2 t_4\,.
    \end{equation}
    Note that $s_{i4}^2$ is fixed by $\s$ for all $i$, and $s_{24}^2 = s_{34}^2$,
    $s_{04}^2 = s_{14}^2$. By applying $\s$ to \eqref{eq:44},
    \begin{equation}
      \s(s_{44})= \e_\s(4) s_{4\hs(4)}=  \e_\s(4) s_{44} = s_{40}^2 t^2_0 + s_{14}^2
      t^2_1 + s_{24}^2 t^2_2 + s_{34}^2 t^2_3+ s_{44}^2 t_4\,.
    \end{equation}
    These equations imply
    \begin{equation}\label{eq:twist_sum}
      (1-\e_\s(4))s_{44} = s_{40}^2 ((t_0 + t_1)- (t^2_0 + t^2_1)) + s_{24}^2 (t_2 +
      t_3-(t^2_2 + t^2_3))\,.
    \end{equation}

    If $t=\diag(\w, \ol\w, 1,1,\pm 1)$, then $(1-\e_\s(4))s_{44} = s_{40}^2 ((\w + \w^4)- (\w^2 + \w^3))$.
    Since the right hand side of this equation is non-zero, $s_{44}\ne 0$ and
    $\e_\s(4)=-1$. Thus, we obtain
    \begin{eqnarray*}
      0 & = & s_{40}^2 (\w + \w^4+ \w^2 + \w^3) + 4 s_{24}^2 + 2 s_{44}^2t_4 \\
        & = & - s_{40}^2  + 4 s_{24}^2 + 2 s_{44}^2 t_4\,
    \end{eqnarray*}
    and hence
    $$
    1 = 2\(  2\frac{s_{24}^2}{s_{04}^2} +  \frac{s_{44}^2}{s_{04}^2} t_4\)\,.
    $$
    Since $s_{j4}/s_{04} \in \BZ$ for all $i$, we find $2 \mid 1$, a contradiction.
    Therefore, $t= \diag(1, 1, \w,\ol \w, \pm 1)$ and the equation \eqref{eq:twist_sum} becomes
    \begin{equation*}
      (1-\e_\s(4))s_{44} = s_{24}^2 ((\w + \w^4)- (\w^2 + \w^3))\,.
    \end{equation*}

    If $s_{24} \ne 0$, then $s_{44} \ne 0$ and $\e_\s(4)=-1$. By the same argument
$$
   \(\frac{s_{24}}{s_{04}}\)^2= 4 + 2 \(\frac{s_{44}}{s_{04}}\)^2 t_4\,.
$$
The integral equation forces $t_4 = 1$ and so $s_{24}^2/2 =2s_{04}^2+ s_{44}^2$. By the unitarity of $s$, we also have
$$
1 = 2 s_{24}^2 + 2s_{04}^2 + s_{44}^2 = \frac{5}{2} s_{24}^2.
$$
This implies $s_{24} =\pm \sqrt{\frac{2}{5}} \in \BQ(\sqrt{5})$ and hence $\sqrt{2} \in \BQ(\sqrt{5})$, a contradiction. Therefore, $s_{24} =0$ and hence $s_{34}=0$. Now, the equation \eqref{eq:44} becomes $s_{44} = 2 s_{04}^2  + s_{44}^2 t_4$. In particular, the integer $s_{44}/s_{04}$ is a root of $ t_4 X^2 -X +2=0$.  This forces $t_4=-1$ and $s_{44}/s_{04} = 1$ or $-2$. By the unitarity of $s$ again, $1 =3 s_{04}^2$ or $1= 6 s_{04}^2$. Both equations imply $\sqrt{3} \in \BQ(\sqrt{5})$, a contradiction.  Now, we can conclude that $5 \nmid n$, so that $n\mid 48 $.

Next we show that $n\mid 24$, \emph{i.e.} $16\nmid n$.

    Suppose to the contrary that $16 \mid n$. Then $\rho \cong (\xi_{16}\o \chi) \oplus \rho_1$ for some subrepresentation
    $\rho_1$ of $\rho$, an irreducible representations $\xi_{16}$ of level 16, and
    an irreducible representation $\chi$ of $\SL{3}$. Then $\deg \xi_{16}=3$, and
    $\deg \chi =1$ and hence they are both even. By tensoring with $\chi\inv$, we may assume $\rho \cong
    \xi_{16} \oplus \rho_1$ for some even subrepresentation $\rho_1$ of $\rho$. The $\ft$-spectrum of $\xi_{16}$ is $\{\w, -\w, \g\}$ for some $\w \in \mu_{16}$ and $\g \in \mu_8$.
    Since  $\deg
    \rho_1 =2$, the level of $\rho_1$ cannot be 16 and so $\pm \w$ are not in the $\ft$-spectrum of $\rho_1$.       Therefore, $\g$ must be an
    eigenvalue of $\rho_1(\ft)$ and hence $\rho_1$ is an irreducible representation
    of level 8. Thus, the $\ft$-spectrum of $\rho_1$ is $\{\g, -\ol\g\}$  (cf. Table \ref{tableA1}).  In particular, $\rho$ is of
    level $n=16$. In view of \eqref{eq:t_form},
    \begin{equation*}
      \rho(\ft)= t =  \diag(\g, \g, \w, -\w,  -\ol\g) \text{ or  }\diag(\w, -\w, \g, \g,  -\ol\g).
    \end{equation*}

    By the (\ref{twisteq}), we find
    \begin{eqnarray*}
      s_{44} &=& {\ol\g}^2 (s_{04}^2 t_0 + s_{14}^2 t_1 + s_{24}^2 t_2 + s_{34}^2 t_3- s_{44}^2 \ol\g)\\
      &=& s_{04}^2 {\ol\g}^2 (t_0 +  t_1) + s_{24}^2 {\ol\g}^2 (t_2 +  t_3)+ s_{44}^2 \g\,.
    \end{eqnarray*}

    If $t =  \diag(\g, \g, \w, -\w,  -\ol\g)$, then
     $s_{44}=2 s_{04}^2 \ol \g+ s_{44}^2 \g$.
    The imaginary parts of both sides of this equation
    imply $ 2 s_{04}^2=
      s_{44}^2$.
    Therefore, $\frac{s_{44}}{s_{04}} = \pm \sqrt{2}$ is not an integer, a
    contradiction.

    If $t =  \diag(\w, -\w, \g, \g, \g')$, then
    $s_{44}=2 s_{24}^2 \ol \g+ s_{44}^2 \g$.
    If $s_{24} \ne 0$, then by the same argument as the preceding case we will arrive the conclusion that $s_{44}/s_{24} \not\in \BQ$. However, this is absurd as both $\frac{s_{44}}{s_{04}}$
    and $\frac{s_{24}}{s_{04}}$ are integers. Therefore, $s_{24}=0$ and hence
    $s_{34}=0=s_{44}$. Orthogonality of $s$ and the action $\s$ imply
    \begin{equation*}
      s_{04}=\pm \frac{1}{\sqrt{2}}, \quad s_{14}=-\e_\s(0)s_{04}, \quad s_{1j} = \e_\s(0) s_{0j}
    \end{equation*}
    for $j = 0,\dots,3$. In particular, $s^2_{12}= s_{02}^2$. Consider the twist equation
    \begin{equation*}
      s_{22}=\g^2 (s_{20}^2 \w - s_{21}^2 \w + (s^2_{22} + s^2_{23})\g) =   (s^2_{22}
      + s^2_{23})\g^3\,.
    \end{equation*}

    This implies $s_{22}=s_{23}=0$ and hence $s_{33}=0$. Consequently,
    the third and the fourth rows of $s$ are multiples of $(1,\e_\s(0),0,0,0)$. This contradicts that $s$ is invertible.

Next we show that $n\mid 12$, \emph{i.e.} $8\nmid n$.  In particular $\Gal(\CC)\cong \BZ_2$.

Since $n\mid 24$,  $\Gal(\BQ_n/\BQ)$ has exponent $2$.
  Galois symmetry and \eqref{eq:t_form} imply
  \begin{equation} \label{eq:t_form2}
   t =\diag(t_0, t_0, t_2,
    t_2, t_4).
  \end{equation}
 In particular, $t$ has at most 3 distinct eigenvalues. By \cite{TubaWenzl}, every irreducible subrepresentation of $\rho$ has degree $\le 3$. Thus,  if $\xi \o \chi$ is isomorphic to an irreducible subrepresentation of $\rho$ for some representations $\xi, \chi$ of ${\rm SL}(2,\Z)$, then $\xi$ or $\chi$ must be linear.

  Suppose that $8\mid n$. Then $\rho \cong (\xi_{8}\o \chi)
    \oplus \rho_1$ for some representations $\xi_8, \chi$ and $\rho_1$ of ${\rm SL}(2,\Z)$ such that
   $\xi_8$ is irreducible of level 8, $\chi$ is irreducible of level 1 or 3, and $\rho_1$ is even.    Since $\deg \xi_8 \ge 2$, $\deg \chi =1$.  Therefore,  $\chi$ is even,  and so is $\xi_8$. By tensoring with $\chi\inv$, we may assume $\rho \cong \xi_{8} \oplus \rho_1$.

     Suppose $\deg \xi_8=3$. Then the eigenvalues of $\xi_8(\ft)$ are
    $\{\w, -\w, \g\}$ for some $\w \in \mu_8$ and $\g \in \mu_4$   (cf. Table \ref{tableA1}).
   In view of \eqref{eq:t_form2}, the $\ft$ spectrum of $\rho_1$ is $\{\w, -\w\}$. In particular,   $\det \rho_1(\ft)= \pm i$ which contradicts that $\rho_1$ is even (cf. Remark \ref{r1}). Therefore, $\deg \xi_8 =2$, and the $\ft$-spectrum of $\xi_8$ is $\{ \g, -\ol\g\}$ for some $\g \in \mu_8$.
    Since $\rho_1(\ft)$ and $\xi_8$ must have a common
    eigenvalue, the level of $\rho_1$ is also a multiple of $8$. By the preceding
    argument, $\rho_1 = \xi'_8 \oplus \rho_2$ for some degree 2 irreducible representation of level 8,  $\xi'_8$, and
    a degree 1 even representation $\rho_2$. However, $\rho_2$ and
    $\xi_8\oplus\xi'_8$ have disjoint $\ft$-spectra, a contradiction. Therefore, $n \mid 12$

    Finally we will show that the
    Frobenius-Schur exponent $N$ must be $2,3,4$ or $6$. Since $N\mid n\mid 12$, it is enough to show $ 4\nmid n$.

    Suppose $4 \mid n$.  We claim that $\rho$ admits a subrepresentation isomorphic to $\xi_4 \o \chi$ for some irreducible representations $\xi_4$ of level 4 and degree $> 1$ and  $\chi \in \Rep(\SL{3})$. Assume the contrary. Since any linear subrepresentation of $\rho$ can only have a level dividing $6$, $\rho$ admits an irreducible subrepresentation $\rho'$  of degree $ > 1$ and level a multiple of 4. Then $\rho'\cong \xi_4 \o \chi$ for some level 4 degree 1 representation $\xi_4$ and an irreducible representation $\chi \in \Rep(\SL{3})$.  Then $\chi$ must be odd since $\xi_4$ is odd. This forces $\chi$ to be of level 3 and degree 2. In particular, $\rho'$ is of level 12 and the $\ft$-spectrum of $\rho'$ is a subset of $\mu_{4*}$. Now, $\rho \cong \rho' \oplus \rho_1$ for some even representation $\rho_1$ of degree 3. By Lemma \ref{l:1}, the level of $\rho_1$ is also a multiple of $4$. Following the same reason, $\rho_1$ admits a degree 2 level 12 even irreducible subrepresentation $\rho''$ with its $\ft$-
spectrum a subset of $\mu_{4*}$. Now, $\rho \cong \rho'\oplus \rho''\oplus \rho_2$ for some degree 1   even representation $\rho_2$. However, $\rho_2$ and $\rho'\oplus \rho''$  have disjoint
    $\ft$-spectra, a contradiction.
This completes the proof.
\end{proof}

\newpage
\appendix
\section{Irreducible Representations of Degree $\leq 4$}
\renewcommand{\thetable}{A.\arabic{table}}
The 12 degree one representations $C_j$ of ${\rm SL}(2,\Z)$, $j=0,1, \dots, 11$ are
defined by
$C_j(\ft) = e^{2\pi j i/12}$. Thus, $C_j$ is even if, and only if, $j$ is even
which is equivalent to the fact that $\ord(C_j) \mid 6$. The $\ft$-spectra of
irreducible representations of degree $\le 4$ and of level $p^\l$ are
illustrated in the following table.
\begin{table}[h!b!p!]
\caption{$\ft$-spectra of level $p^\l$ irreducible representations of degree
$\le 4$ }
\begin{tabular}{cccl}
\hline
degree & parity & level  & $\ft$-spectra \\
\hline
2 & even & 2 & $\{1, -1\}$ \\
 & odd & 3  & $\{e^{2 \pi r i/3}, e^{-2 \pi (r+1)  i/3}\}$, $r=0,1,2$ \\
 & odd & 4 & $\{i, -i\}$ \\
 & odd & 5 & $\{e^{2 \pi  i/5}, e^{-2 \pi  i/5}\}, \{e^{4 \pi  i/5}, e^{-4 \pi
i/5}\}$ \\
 & even & 8 & $\{e^{5 \pi  i/4},e^{7 \pi  i/4}\}$, $\{e^{ \pi  i/4},e^{3 \pi
i/4}\}$ \\
 & odd & 8 & $\{e^{3 \pi  i/4},e^{5 \pi  i/4}\}, \{e^{7 \pi  i/4},e^{ \pi
i/4}\}$ \\
\hline
3 & even & 3 & $\{e^{2 \pi (r+1) i/3}, e^{2 \pi (r+2)  i/3}, e^{2 \pi r
i/3}\}$, $r=0,1,2$ \\
 & odd & 4 & $\{i, -1, 1\}$, $\{-i, 1, -1\}$\\
 & even & 4 &  $\{-1, -i, i\}$, $\{1, i, -i\}$ \\
 & even & 5 & $\{1, e^{2 \pi r  i/5}, e^{-2 \pi r  i/5}\}$, $r=1,2$ \\
 & even & 7 & $\{e^{4 \pi   i/7}, e^{2 \pi   i/7}, e^{8 \pi   i/7}\}$,\\
      &  &  &  $\{e^{-4 \pi   i/7}, e^{-2 \pi   i/7}, e^{-8 \pi   i/7}\}$\\
& odd & 8 & $\{-1, -e^{ \pi i/4}, e^{\pi i/4} \}$, $\{1, e^{ \pi i/4}, -e^{\pi
i/4} \}$\\
      &  &  &  $\{-1, -e^{ 3\pi i/4}, e^{3 \pi i/4} \}$, $\{1, e^{ 3\pi i/4},
-e^{3 \pi i/4} \}$\\
& even & 8 & $\{-i, -e^{ \pi 3i/4}, e^{\pi 3i/4} \}$, $\{i, e^{ \pi 3i/4},
-e^{\pi 3i/4} \}$\\
      &  &  &  $\{i, -e^{\pi i/4}, e^{ \pi i/4} \}$, $\{-i, e^{ \pi i/4},
-e^{\pi i/4} \}$ \\
& odd & 16 & $\{-e^{\pi i/4}, e^{ \pi i/8}, -e^{\pi i/8} \}$, $\{e^{\pi i/4},
-e^{ \pi i/8}, e^{\pi i/8} \}$\\
      &  &  &  $\{e^{\pi i/4}, e^{ 5\pi i/8}, -e^{5 \pi i/8} \}$, $\{-e^{\pi
i/4}, -e^{ 5\pi i/8}, e^{5 \pi i/8} \}$\\
        &  &  &  $\{-e^{\pi 3 i/4}, e^{ 3 \pi i/8}, -e^{ 3 \pi i/8} \}$,
$\{e^{\pi 3 i/4}, -e^{ 3 \pi i/8}, e^{ 3 \pi i/8} \}$\\
        &  &  &  $\{e^{3 \pi i/4}, -e^{7\pi i/8}, e^{7 \pi i/8} \}$, $\{-e^{3
\pi i/4}, -e^{7 \pi i/8}, e^{7 \pi i/8} \}$\\
& even & 16 & $\{-e^{3\pi i/4}, e^{ 5 \pi i/8}, -e^{\pi 5i/8} \}$, $\{e^{3\pi
i/4}, -e^{ 5 \pi i/8}, e^{\pi 5i/8} \}$\\
      &  &  &  $\{e^{3\pi i/4}, -e^{ \pi i/8}, e^{ \pi i/8} \}$, $\{-e^{3 \pi
i/4}, e^{\pi i/8}, -e^{\pi i/8} \}$\\
        &  &  &  $\{-e^{\pi i/4}, -e^{ 7 \pi i/8}, e^{ 7 \pi i/8} \}$, $\{e^{\pi
i/4}, e^{ 7 \pi i/8}, -e^{ 7 \pi i/8} \}$ \\
        &  &  &  $\{-e^{ \pi i/4}, e^{3\pi i/8}, -e^{3 \pi i/8} \}$, $\{e^{ \pi
i/4}, -e^{3 \pi i/8}, e^{3 \pi i/8} \}$\\
\hline
4 & odd & 5 & $\{e^{2 \pi i/5}, e^{4 \pi i/5}, e^{6 \pi i/5}, e^{8 \pi i/5}\} $
\\
&even &5 &$\{e^{2 \pi i/5}, e^{4 \pi i/5}, e^{6 \pi i/5}, e^{8 \pi i/5}\} $ \\
 & odd & 7 & $\{1,e^{2 \pi i/7}, e^{8 \pi i/7},e^{4 \pi i/7}\}$\\
  & odd & 7 & $\{1,e^{12 \pi i/7}, e^{6 \pi i/7},e^{10 \pi i/7}\}$\\
 & odd & 8 &  $\{e^{\pi i/4},e^{3\pi i/4},e^{5\pi i/4},e^{7\pi i/4}\}$ \\
 & even & 8 &  $\{e^{\pi i/4},e^{3\pi i/4},e^{5\pi i/4},e^{7\pi i/4}\}$ \\
 & odd & 9 & $\{e^{2 \pi i(\frac{1}{9}+\frac{r}{3})},  e^{2 \pi
i(\frac{4}{9}+\frac{r}{3})}, e^{2 \pi i(\frac{7}{9}+\frac{r}{3})}, e^{2 \pi
i(\frac{1}{3}+\frac{r}{3})}\}$, $r=0,1,2$ \\
 &  & & $\{e^{2 \pi i(\frac{8}{9}+\frac{r}{3})},  e^{2 \pi
i(\frac{5}{9}+\frac{r}{3})}, e^{2 \pi i(\frac{2}{9}+\frac{r}{3})}, e^{2 \pi
i(\frac{2}{3}+\frac{r}{3})}\}$, $r=0,1,2$ \\
 & even & 9 & $\{e^{2 \pi i(\frac{1}{9}+\frac{r}{3})},  e^{2 \pi
i(\frac{4}{9}+\frac{r}{3})}, e^{2 \pi i(\frac{7}{9}+\frac{r}{3})}, e^{2 \pi
i(\frac{1}{3}+\frac{r}{3})}\}$, $r=0,1,2$ \\
 &  & & $\{e^{2 \pi i(\frac{8}{9}+\frac{r}{3})},  e^{2 \pi
i(\frac{5}{9}+\frac{r}{3})}, e^{2 \pi i(\frac{2}{9}+\frac{r}{3})}, e^{2 \pi
i(\frac{2}{3}+\frac{r}{3})}\}$, $r=0,1,2$ \\
\end{tabular}
\label{tableA1}
\end{table}

\clearpage
\pagestyle{plain}
{\raggedright\printbibliography}

\end{document}